\documentclass[a4paper,11pt,reqno,noindent]{amsart}
\usepackage[centertags]{amsmath}
\usepackage{amsfonts,amssymb,amsthm,dsfont,cases,amscd,esint,enumerate, stmaryrd}
\usepackage{upgreek}
\usepackage{mathrsfs}
\usepackage[T1]{fontenc}
\usepackage[english]{babel}
\usepackage[applemac]{inputenc}
\usepackage{newlfont}
\usepackage{color}
\usepackage[body={17cm,21.5cm}, top = 1in, bottom = 1.5in, centering]{geometry} 
\usepackage{fancyhdr}
\usepackage{enumerate}
\usepackage{comment}

\usepackage{mathtools}
\usepackage[scr=boondoxo]{mathalfa}

\theoremstyle{plain}
\newtheorem{theor}{Theorem}[section]
\newtheorem{lem}[theor]{Lemma}
\newtheorem{prop}[theor]{Proposition}
\newtheorem{cor}[theor]{Corollary}

\theoremstyle{definition}

\newtheorem{rem}[theor]{Remark}

\mathchardef\emptyset="001F
\numberwithin{equation}{section}

\newcommand{\R}{\mathbb R}

\newcommand{\Z}{\mathbb Z}

\newcommand{\T}{\mathbb T}

\newcommand{\e}{\upvarepsilon}
\newcommand{\Pc}{\mathcal{P}}
\newcommand{\Pp}{\mathbb{P}}

\newcommand{\Id}{\operatorname{Id}}

\newcommand{\E}{\mathbb{E}}
\newcommand{\supess}{\operatorname{supess}}
\newcommand{\supp}{\operatorname{supp}}
\newcommand{\Var}{\operatorname{Var}}

\newcommand{\loc}{{\operatorname{loc}}}

\newcommand{\Cov}{{\operatorname{Cov}}}

\newcommand{\step}[1]{\noindent \textit{Step} #1.}

\let\le\leqslant
\let\leq\leqslant
\let\ge\geqslant
\let\geq\geqslant

\usepackage[colorlinks,citecolor=black,urlcolor=black]{hyperref}
\usepackage{tikz}
\definecolor{revisiongreen}{rgb}{0.00,0.50,0.00}
\usepackage{graphicx}
\usepackage{pgfplots}
\usepgfplotslibrary{groupplots}
\pgfplotsset{
paperplot/.style={
grid=both,
minor grid style={gray!15},
major grid style={gray!25},
tick align=outside,
label style={font=\footnotesize},
tick label style={font=\scriptsize},
title style={font=\footnotesize},
},
scorecurve/.style={thick, blue},
minmark/.style={
only marks,
mark=o,
mark size=2.2pt,
mark options={draw=blue, fill=white, line width=0.9pt},
},
maxmark/.style={
only marks,
mark=triangle,
mark size=2.6pt,
mark options={draw=blue, fill=white, line width=0.9pt},
},
samplemark/.style={
only marks,
mark=o,
mark size=2.2pt,
mark options={draw=blue, fill=white, line width=0.9pt},
},
}
\usetikzlibrary{fit}
\usetikzlibrary{shapes.geometric}
\usetikzlibrary{calc}

\title[Lost in the middle]{Kinetic theory for Transformers\\and the lost-in-the-middle phenomenon}

\author[M.~Duerinckx]{Mitia Duerinckx}
\author[B.~Geshkovski]{Borjan Geshkovski}
\author[S.~Rossi]{Stefano Rossi}

\begin{document}

\begin{abstract}
We study causal self-attention dynamics---a toy model for decoder Transformers---which we interpret as a non-exchangeable interacting particle system. By a cumulant expansion adapted to the triangular causal dependency structure of the model, combined with Glauber-calculus estimates on higher-order correlations, we prove a quantitative mean-field limit result and a next-order characterization of correlations. For iid uniformly distributed tokens, the limiting correlation equation can be solved in closed form, yielding a rigorous explanation of the empirically observed \emph{lost-in-the-middle} phenomenon: the token retrieval profile, as a function of the source position in the prompt, is $\mathsf{U}$-shaped, with primacy, recency, and a unique interior minimum under an explicit smallness condition. Crucially, the dependence on the source position is invisible at mean-field order and only emerges through the leading $O(N^{-1})$ correlations.
\end{abstract}

\maketitle
\setcounter{tocdepth}{1}
\allowdisplaybreaks

\section{Introduction and main results}\label{sec:intro}

Large language models display several long-context effects, one of the
cleanest being the dependence of the output on the position of relevant information
inside the prompt. A particularly simple protocol consists in inserting a single
relevant fact at a chosen location, filling the rest of the prompt with
distractors, and then asking a retrieval or question-answering task whose correct
answer depends only on that fact. Repeating this experiment across many prompts and across a wide class of models reveals a consistent pattern: accuracy
is high near the beginning of the prompt, high again near the end, and lower in
the middle. This empirically observed pattern goes by the name of \emph{lost-in-the-middle} \cite{liu2023lost}; see  Figure~\ref{fig:litm-paper}.

\begin{figure}[h]
\centering
\begin{minipage}[c]{.30\textwidth}
\centering
\includegraphics[width=\linewidth]{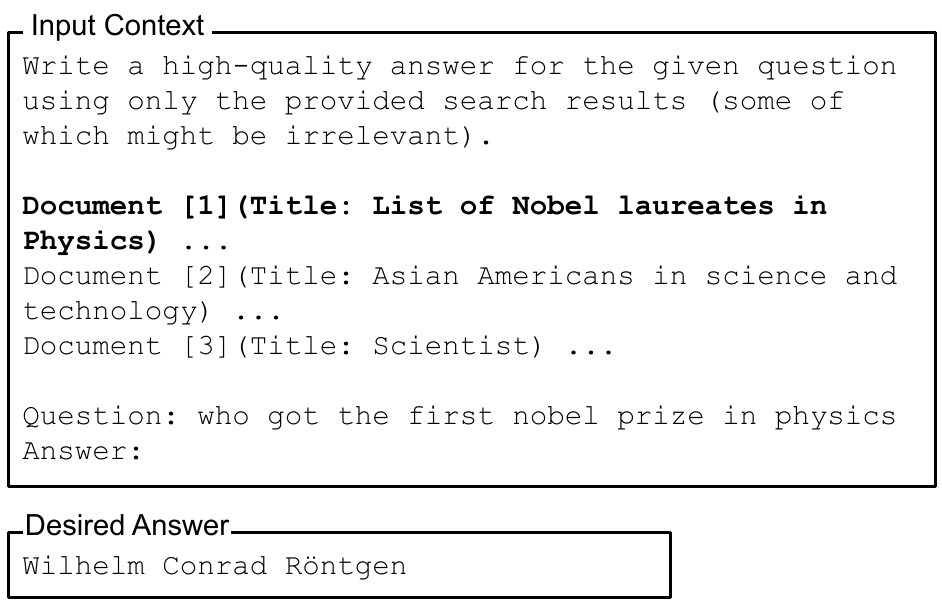}\\[-.35em]
\includegraphics[width=\linewidth]{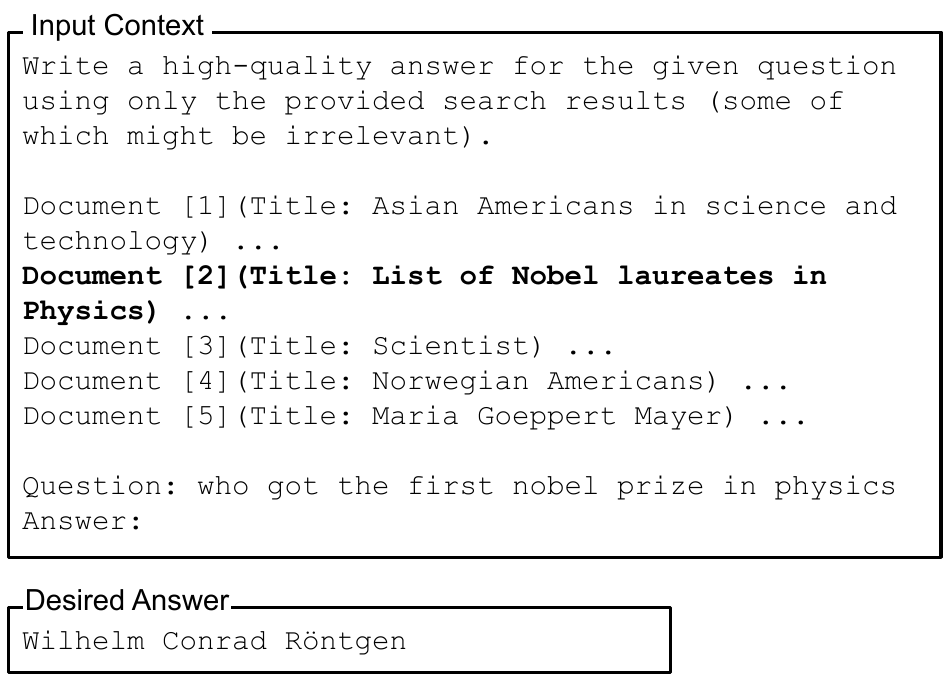}
\end{minipage}
\hspace{.04\textwidth}
\begin{minipage}[c]{.36\textwidth}
\centering
\includegraphics[width=\linewidth]{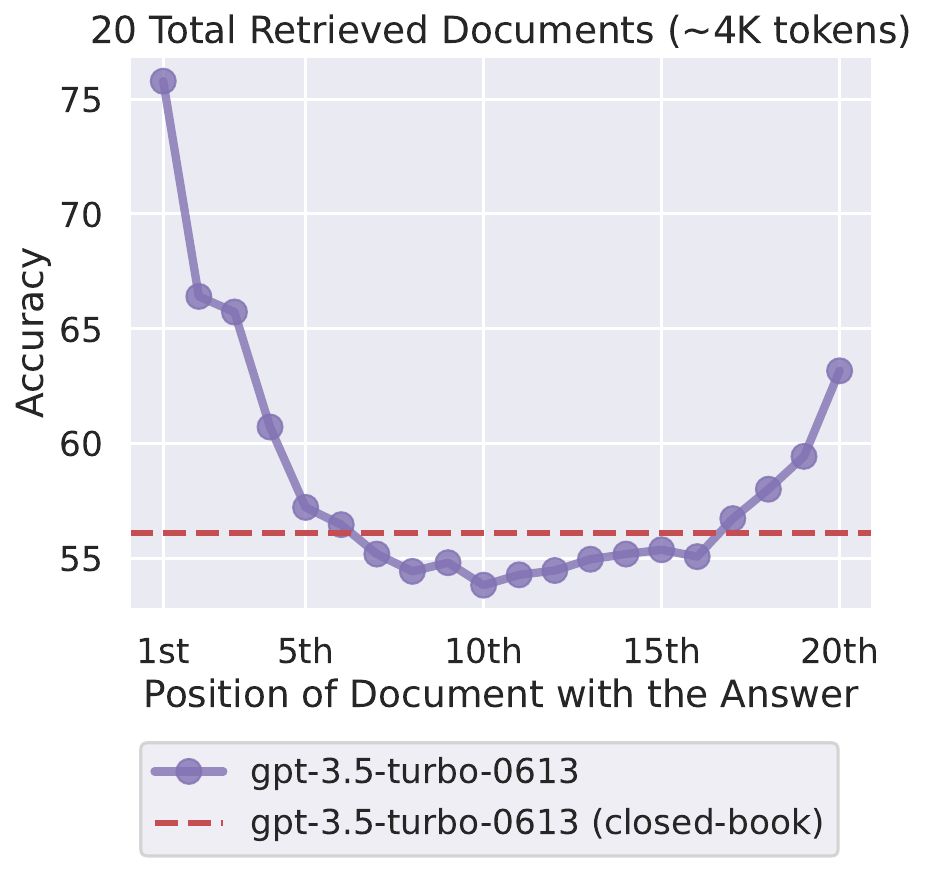}
\end{minipage}
\caption{The experiment in~\cite{liu2023lost} that motivates the paper. The two panels on the left show how the answer-containing document is moved through the prompt and how the amount of surrounding text is changed. The curve on the right is the empirical lost-in-the-middle profile: retrieval is strongest near the beginning and the end of the context, and weakest in the middle. Reproduced from Figures~3, 4, and~1 of~\cite{liu2023lost}, respectively, with permission from the authors.}
\label{fig:litm-paper}
\end{figure}

A convenient framework for addressing this question from a mathematical lens is to view Transformers \cite{vaswani}---the neural network architecture underlying large language models---as interacting particle
systems~\cite{geshkovski2025mathematical}. This has already proved fruitful
for encoder Transformers, which are exchangeable particle systems, where empirically observed representation collapse phenomena \cite{zhou2021deepvit, dong2021attention, noci2022signal, queipo-de-llano2026attention} can be rigorously proven and seen as clustering of the particles over time~\cite{GLPR23, geshkovski2025mathematical, chen2025quantitative, polyanskiy2025synchronization, geshkovski2024dynamic, BrunoPasqualottoAgazzi2024}. Results of this nature have since extended to substantially more general settings~\cite{bruno2025a, 
KoubbiGeshkovskiRigollet2026, FedorovSanderElieMarionLauriere2026, karagodin2024clustering, alvarez2026perceptrons, geshkovski2024measure, castin2025unified}. 

We leverage this perspective for decoder
Transformers, for which less is known. More precisely,
we consider a minimal model on the one-dimensional torus $\T:=\R/2\uppi\Z$, structurally faithful to the ones used in the experiments of~\cite{liu2023lost}.
Namely, at layer $t\geq0$, the $j$-th token embedding, for $ 1\le j\le N$, evolves as
\begin{equation}\label{eq:gpt-dyn-alibi}
\frac{\mathrm d}{\mathrm dt}\uptheta_j(t)
=
\frac{1}{\mathscr{Z}_{N,j}}
\sum_{k=1}^{j-1}e^{-\frac{\uplambda}{N}(j-k)}\mathsf{w}_{\upbeta}'\left(\uptheta_j(t)-\uptheta_k(t)\right),
\end{equation}
with normalization
\begin{equation}\label{eq:partition-function}
\mathscr{Z}_{N,j}:=\sum_{k=1}^{j-1}e^{-\frac{\uplambda}{N}(j-k)},
\end{equation}
and interaction kernel
\begin{equation}\label{eq:wbeta}
\mathsf{w}_{\upbeta}(\uptheta):=e^{\upbeta\cos\uptheta},
\end{equation}
for fixed $\upbeta>0$ and $\uplambda\in\R$.
This corresponds to the ``USA'' dynamics of~\cite{GLPR23}, but in a causal
form, as in decoder architectures, and it incorporates the positional encoding factor
$e^{-\uplambda(j-k)/N}$ known as ALiBi~\cite{Press2021ALiBi}.
In particular, causality makes the system non-exchangeable. We focus on this
intentionally caricatural setting for clarity and simplicity: it is not meant to represent full Transformers, which are high-dimensional and involve trainable matrices (see \eqref{eq:full-dyn}), but models of this type have been remarkably predictive of the behavior of trained  Transformers, as illustrated for instance in~\cite[Figure~1]{geshkovski2025mathematical}, and they are rigorously justified at the initialization of training~\cite{KoubbiGeshkovskiRigollet2026,BrunoPasqualottoAgazzi2024,FedorovSanderElieMarionLauriere2026}. The reduced system~\eqref{eq:gpt-dyn-alibi} already retains the two mechanisms central to the present study: early tokens are repeatedly reused by later ones, while recent tokens are favored by the positional bias. Our analysis also extends to general smooth interaction kernels and causal positional encodings, see Section~\ref{sssec:positional-encodings}.

To connect the dynamics with plots such as Figure~\ref{fig:litm-paper}, we use
a minimal decoder, which is a stripped-down version of the retrieval task in \cite{liu2023lost}: one singles out a relevant source position and asks
whether the final output token recovers it. Fix a vocabulary size $M\ge2$, say $\mathscr V=\{0,1,\ldots,M-1\}$, and let
\[
\upvartheta_m:=\frac{2\uppi m}{M}\in\T,\qquad m\in\mathscr V.
\]
We encode an input token $m_i$ at position $i$ by $\uptheta_i(0)=\upvartheta_{m_i}$.
Given the last hidden state $\uptheta_N(t)$, the decoder returns the nearest codeword
\[
\hat m_N(t)
:=\arg\min_{m\in\mathscr V} \left|\uptheta_N(t)-\upvartheta_m\right|_{\T}.
\]
Since one should think of many prompts, or equivalently of many possible initial conditions, the relevant observable is an averaged retrieval score over a statistical ensemble of prompts.
For a distinguished source position $1\le i_0=\lceil \upsigma_0N\rceil\le N$, with
$\upsigma_0\in(0,1]$, we therefore define the prediction accuracy by
\begin{equation}\label{eq:Acc-def}
\Pp\left[\hat m_N(t)=m_{i_0}\right]
=
\E\left[
\mathds1_{\left\{\left|\uptheta_N(t)-\uptheta_{i_0}(0)\right|_{\T}\le \frac{\uppi}{M}\right\}}
\right],
\end{equation}
where the expectation is taken with respect to the random ensemble of prompts $\{\uptheta_i(0)\}_i$.
For analytical convenience, we mollify the characteristic function,
e.g.\@ using a
periodic Gaussian,
thus defining
the
{\it soft accuracy}
\begin{equation}\label{eq:Acc-soft}
\mathscr{A}_N(t,\upsigma_0)
:=
\E
\sum_{k\in\Z}
\exp\left(
-\frac{M^2}{2\uppi^2}
\left(\uptheta_N(t)-\uptheta_{i_0}(0)-2\uppi k\right)^2
\right).
\end{equation}
By the Poisson summation formula, this admits the Fourier representation
\begin{equation}\label{eq:Acc-soft-Fourier}
\mathscr{A}_N(t,\upsigma_0)
=
\frac{\sqrt{\uppi/2}}{M}
\sum_{n\in\Z}
e^{-\frac{\uppi^2}{2M^2}n^2}
\E \left[e^{in(\uptheta_N(t)-\uptheta_{i_0}(0))}\right].
\end{equation}
Our goal is to prove that this quantity satisfies the $\mathsf{U}$-shape observed in~\cite{liu2023lost},
see Figure~\ref{fig:litm-paper}.
This amounts to analyzing the large-$N$ behavior of the expectations $\E[e^{in(\uptheta_N(t)-\uptheta_{i_0}(0))}]$ as the source location $\upsigma_0\approx\frac{i_0}{N}$ varies.

We approach this problem using tools from kinetic theory.
A well-developed framework exists for deriving continuum descriptions from interacting particle systems: in the mean-field regime, this originates in propagation of chaos and the associated McKean--Vlasov limits; see e.g.\@ the reviews~\cite{Golse2016-rev,Jabin2017-rev}.
More recently, the theory has been extended, in combination with graph-limit techniques, to non-exchangeable systems; see e.g.~\cite{JabinPoyatoSoler2025,paultrel2026}.
The model~\eqref{eq:gpt-dyn-alibi} considered here, however, falls outside these frameworks: the causal interaction weights
\[\upomega_{j,k}:=\frac{1}{\mathscr{Z}_{N,j}}e^{-\lambda(j-k)/N}\mathds1_{k<j}\]
do not satisfy the standard boundedness requirement of e.g.~\cite[Assumption~(3)]{JabinPoyatoSoler2025}, but instead obey
\begin{equation}\label{eq:min-requ-jabinsoler-0}
\sup_{1\le k\le N}\sum_{j=1}^N\upomega_{j,k}\simeq\log N,
\end{equation}
reflecting the logarithmically growing cumulative influence of early indices on later particles.
Closely related systems have recently been studied in~\cite{DuJiangLi,WangZhao1}.

Beyond this structural singularity of the model, the retrieval observable~\eqref{eq:Acc-soft-Fourier} exhibits a more fundamental feature: its dependence on the source position does not appear at the mean-field level. Capturing it therefore requires going beyond propagation of chaos and resolving the leading $O(N^{-1})$ correlations between a fixed source token and a later token, which is the main objective of this work.
For exchangeable systems, the analysis of such correlations has seen significant recent progress, with both trajectorial approaches~\cite{MD-21} and hierarchy-based methods~\cite{PPS-19,Hess_Childs_2023,DJ_25}, whereas results for non-exchangeable systems remain scarce.
At the fluctuation level, recent works~\cite{ShkolnikovYeung,WangZhao2} study global $O(N^{-1/2})$ fluctuations for non-exchangeable diffusions. While such fluctuation results encode spatially averaged $O(N^{-1})$ pair correlations, our retrieval problem requires a finer, source-resolved description: we aim to characterize the $O(N^{-1})$ correlation between a token at a prescribed macroscopic position $\upsigma_0$ and a later token, retaining its full dependence on $\upsigma_0$.
Moreover, these fluctuation results~\cite{ShkolnikovYeung,WangZhao2} are proved only for stochastic diffusive dynamics, whereas our model~\eqref{eq:gpt-dyn-alibi} is deterministic.
To address this problem, we adapt the trajectorial approach of~\cite{MD-21}, based on Glauber calculus with respect to initial data. A key feature of our analysis is that the singular scaling~\eqref{eq:min-requ-jabinsoler-0} induces nonstandard corrections to the propagation of chaos.

\subsection{Main results}
We define the empirical measure
\begin{equation}\label{eq:emp-meas}
\upmu_N(t,\upsigma,\uptheta)=\frac1N\sum_{j=1}^N\updelta_{(\frac jN,\uptheta_j(t))}(\upsigma,\uptheta)\in L^\infty(\mathbb{R}_{\geqslant0};\Pc((0,1)\times\T)),
\end{equation}
where the rescaled index $\upsigma=\frac jN\in(0,1]$ plays the role of a continuous reading-cursor variable in the limit. Next, motivated by the positional encoding factor $\frac1{\mathscr{Z}_{N,j}}e^{-\uplambda(j-k)/N}\mathds1_{k<j}$ in~\eqref{eq:gpt-dyn-alibi}, we introduce the associated limiting directed graphon
\begin{equation}\label{eq:kernel-Klambda}
\mathsf{k}_{\uplambda}(\upsigma,\upsigma')
:=
\frac{\uplambda e^{-\uplambda(\upsigma-\upsigma')}}{1-e^{-\uplambda\upsigma}}\,\mathds1_{\upsigma'<\upsigma},
\end{equation}
which is understood by continuity as $\upsigma^{-1}\mathds1_{\upsigma'<\upsigma}$
when $\uplambda=0$.
We denote by $\hat f(n):=\int_\T e^{-in\uptheta}f(\uptheta)\,\mathrm d\uptheta$ the Fourier transform on $\T$, for $n\in\Z$, and write $\langle n\rangle:=\sqrt{1+n^2}$.

Our first result describes the mean-field limit of the system as $N\to\infty$.
We stress that even the qualitative convergence in part~(i) does not follow from existing mean-field results for non-exchangeable systems, since the interaction weights in~\eqref{eq:gpt-dyn-alibi} fail to satisfy the boundedness assumption of~\cite[Assumption~(3)]{JabinPoyatoSoler2025}; see~\eqref{eq:min-requ-jabinsoler-0}.

\begin{theor} \label{th:mf-gpt}\
\begin{enumerate}[(i)]
\item\emph{Qualitative mean-field limit:}\\
If initially $\upmu_N|_{t=0}\overset*\rightharpoonup f_\circ$ in $\Pc((0,1]\times\T)$, then we have
\[\upmu_N(t)\overset*\rightharpoonup f(t)\qquad\text{for all $t\ge0$,}\]
where the limit~$f$ is the unique weak solution in $L^\infty(\mathbb{R}_{\geqslant0};\Pc((0,1]\times\T))$ of the kinetic equation
\begin{equation}\label{eq:mfl-lambda}
\left\{\begin{array}{l}\displaystyle
\partial_t f(t,\upsigma,\uptheta)
+\partial_\uptheta\left(
 f(t,\upsigma,\uptheta)\int_0^\upsigma \mathsf{k}_{\uplambda}(\upsigma,\upsigma')(\mathsf{w}_{\upbeta}'\ast_\uptheta f)(t,\upsigma',\uptheta)\,\mathrm d\upsigma'
\right)=0,\\
f|_{t=0}=f_\circ.
\end{array}\right.
\end{equation}
\item\emph{Error estimates:}\\
Assume that the initial data $(\uptheta_j^0)_{1\le j\le N}$ are independent and that their distribution converges to some limit profile $f_\circ\in C^0([0,1];\Pc(\T))$ in the following sense: for some $\updelta>0$ and~$C,\upgamma<\infty$,
\begin{equation}\label{eq:init-conv}
\left|\E\left[e^{in\uptheta^0_{\lceil N\upsigma\rceil}}\right]-\hat f_\circ\left(\upsigma,n\right)\right|\le CN^{-\updelta}\langle n\rangle^\upgamma\quad\text{for all $\upsigma\in(0,1]$ and $n\in \Z$}.
\end{equation}
Then, for all $t\ge0$, $\upvarphi\in C^\infty([0,1]\times\T)$, and $m>\upgamma\vee2+\frac12$,
\begin{equation}\label{eq:conv-rate-muNf}
\E\left[\bigg|\int_{(0,1]\times\T}\upvarphi\left(\upmu_N(t)-f(t)\right)\bigg|^2\right]^\frac12\lesssim N^{-\updelta\wedge\frac12}e^{C t}\|\upvarphi\|_{L^\infty([0,1];W^{m,\infty}(\T))}.
\end{equation}
\end{enumerate}
\end{theor}

We illustrate the meaning of the assumptions on the initial data with a simple position-dependent prompt generator.
Fix a vocabulary \(\mathscr V=\{0,1,\ldots,M-1\}\), encode each token $m\in\mathscr V$ by $\upvartheta_m:=\frac{2\uppi m}{M}\in\T$, and let $p_m\in C^1([0,1])$ be such that $p_m(\upsigma)\ge0$ and $\sum_{m\in\mathscr V}p_m(\upsigma)=1$.
These token frequencies are allowed to depend smoothly on the position in the sequence. For instance, early, middle, and late positions may favor different subsets of the vocabulary: early positions may favor instruction-like tokens, while later positions favor answer-related tokens.
Now consider initial prompts constructed as follows: for each prompt length $N$, sample independently $m_i\sim p_\cdot(i/N)$ and set $\uptheta_i(0)=\upvartheta_{m_i}$.
The variables \(\uptheta_i(0)\) are then independent, but in general not identically distributed, and the assumption~\eqref{eq:init-conv} follows directly from the \(C^1\)-regularity of~\((p_m)_m\), with \(f_\circ(\upsigma,\uptheta)=\sum_{m\in\mathscr V}p_m(\upsigma)\updelta_{\upvartheta(m)}(\uptheta)\).
In the simplest homogeneous version of this baseline, the initial tokens may be taken iid and uniformly distributed on~$\T$, corresponding to $f_\circ\equiv1$.

As is typical for mean-field limits of non-exchangeable systems (see e.g.~\cite{JabinPoyatoSoler2025,paultrel2026}), the limit equation~\eqref{eq:mfl-lambda} takes the form of a ``layered'' McKean--Vlasov equation. In the present case, this structure reflects the causal nature of the dynamics: the macroscopic variable $\upsigma$ quantifies how much of the past is accessible to a given token, while the kernel $\mathsf{k}_{\uplambda}$ encodes the positional bias.
However, the mean-field limit alone does not yield useful information on the accuracy function~\eqref{eq:Acc-soft} in the retrieval task. Indeed, the relevant signal is carried by time correlations $\Cov(e^{in\uptheta_N(t)},e^{-in\uptheta_{i_0}(0)})$, which vanish at the level of the mean-field description. Capturing the spatial dependence of the accuracy function therefore requires analyzing the next order in the propagation of chaos.

To this end, for a test function $\upvarphi\in C^\infty(\T)$, we introduce the autocorrelation $A_\upvarphi^N$ and the cross-correlation $C_\upvarphi^N$ as follows: for $t\ge0$, $\upsigma,\upsigma_0\in(0,1]$, and $\uppsi\in C^\infty(\T)$,
\begin{eqnarray}
\int \uppsi(\uptheta)\,A_{\upvarphi}^N(t,\upsigma,\uptheta)\,\mathrm d\uptheta&:=& \Cov\left(\uppsi\big(\uptheta_{\lceil N\upsigma\rceil}(t)\big),\upvarphi\big(\uptheta^0_{\lceil N\upsigma\rceil}\big)\right),\label{eq:def-Covn}\\
\int \uppsi(\uptheta)\,C_{\upvarphi}^N(t,\upsigma,\uptheta;\upsigma_0)\,\mathrm d\uptheta&:=& \Cov\left(\uppsi\big(\uptheta_{\lceil N\upsigma\rceil}(t)\big),\upvarphi\big(\uptheta^0_{\lceil N\upsigma_0\rceil}\big)\right)\,\mathds1_{\lceil N\upsigma_0\rceil<\lceil N\upsigma\rceil}.\label{eq:def-Covn-2}
\end{eqnarray}
The autocorrelation $A_\upvarphi^N$ quantifies the memory of a token at position $\upsigma$ at time $t$ with respect to its own initial value; it is of order $O(1)$. By contrast, the cross-correlation $C_\upvarphi^N$ measures the dependence of a token at position $\upsigma$ at time $t$ with respect to the initial value of an earlier token at a different position $\upsigma_0<\upsigma$; it is of order $O(N^{-1})$. It is precisely this cross-correlation that carries the information on the spatial dependence of the accuracy function. In the following result, we characterize the leading behavior of both $A_\upvarphi^N$ and $C_\upvarphi^N$.

\begin{theor}\label{th:lost}
Assume that the initial data $(\uptheta_j^0)_{1\le j\le N}$ are independent and that their distribution converges to some limit profile $f_\circ\in C^0([0,1];\Pc(\T))$ in the sense of~\eqref{eq:init-conv}, for some $\updelta>0$.
Then, given~$\upzeta\le\updelta$ with $\upzeta<1$, we have for all $t\ge0$, $\upsigma,\upsigma_0\in(0,1]$, $\upsigma_0<\upsigma$, and frequency $n\in \Z$,
\begin{eqnarray}
\left|\left(\hat A_\upvarphi^N-\hat A_\upvarphi\right)(t,\upsigma,n)\right|
&\lesssim_{\upzeta,\upvarphi}&\frac1{N^{\upzeta}\upsigma^2}\langle n\rangle^Ce^{C t},\label{eq:estim-ANA}\\
\left|\left(N\hat C_\upvarphi^N-\hat C_\upvarphi\right)(t,\upsigma,n;\upsigma_0)\right|
&\lesssim_{\upzeta,\upvarphi}&\frac1{N^{\upzeta}\upsigma_0^2}\langle n\rangle^Ce^{C t},\label{eq:estim-CNC}
\end{eqnarray}
where $A_\upvarphi,C_{\upvarphi}$ are the unique solutions of the limit equations
\begin{align}
&\left\{\begin{array}{l}\displaystyle
\partial_tA_\upvarphi(t,\upsigma,\uptheta)
=-\partial_\uptheta\left(A_\upvarphi(t,\upsigma,\uptheta)\int_0^\upsigma \mathsf{k}_{\uplambda}(\upsigma,\upsigma')(\mathsf{w}_{\upbeta}'\ast_\uptheta f)(t,\upsigma',\uptheta)\,\mathrm d\upsigma'\right),\\[3mm]\displaystyle
A_\upvarphi(0,\upsigma,\uptheta)=f_\circ(\upsigma,\uptheta)\left(\upvarphi(\uptheta)-\int_\T \upvarphi f_\circ(\upsigma,\cdot)\,\mathrm d\uptheta\right),
\end{array}\right.
\label{eq:Aphi-lambda}\\[2mm]
&\left\{\begin{array}{l}\displaystyle
\partial_tC_\upvarphi(t,\upsigma,\uptheta;\upsigma_0)
=-\partial_\uptheta\left( f(t,\upsigma,\uptheta)\int_{\upsigma_0}^\upsigma \mathsf{k}_{\uplambda}(\upsigma,\upsigma')(\mathsf{w}_{\upbeta}'\ast_\uptheta C_\upvarphi)(t,\upsigma',\uptheta;\upsigma_0)\,\mathrm d\upsigma'\right)\\[3mm]\displaystyle
\hspace{3cm}-\partial_\uptheta\left( C_\upvarphi(t,\upsigma,\uptheta;\upsigma_0)\int_0^\upsigma \mathsf{k}_{\uplambda}(\upsigma,\upsigma')(\mathsf{w}_{\upbeta}'\ast f)(t,\upsigma',\uptheta)\,\mathrm d\upsigma'\right)\\[3mm]\displaystyle
\hspace{3cm}-\partial_\uptheta\left(f(t,\upsigma,\uptheta)\,\mathsf{k}_{\uplambda}(\upsigma,\upsigma_0)(\mathsf{w}_{\upbeta}'\ast_\uptheta A_\upvarphi)(t,\upsigma_0,\uptheta)\right),\\
C_\upvarphi(0,\upsigma,\uptheta;\upsigma_0)=0.
\end{array}\right.
\label{eq:Cphi-lambda}
\end{align}
\end{theor}

The limiting autocorrelation $A_\upvarphi$ is transported by the mean-field flow, while the cross-correlation~$C_\upvarphi$ evolves according to the linearization of the mean-field dynamics, with a forcing term determined by~$A_\upvarphi$.
This forcing injects the fluctuation carried by the source token at position $\upsigma_0$ into the dynamics
of~$C_\upvarphi$, and is weighted
by the limiting graphon $\mathsf{k}_{\uplambda}(\upsigma,\upsigma_0)$, thereby encoding the spatial influence of the source point. This mechanism is ultimately responsible for the dependence
of the retrieval profile on the source position $\upsigma_0$.

The proof relies on a non-hierarchical cumulant expansion adapted to the triangular causal structure of the dynamics.
More precisely, differentiating a two-point correlation in time produces third-order cumulants, so that no closed equation is available at the level of two-point correlations.
To control these higher-order terms, we appeal to Glauber calculus with respect to the initial data, following the approach of~\cite{MD-21}.
In essence, quantitative bounds on first- and second-order Glauber derivatives yield sharp control of the third-cumulant remainder, allowing one to truncate the expansion at the level of two-point correlations without resorting to a BBGKY hierarchy.

Compared with previous uses of Glauber calculus to quantify corrections to mean-field limits, notably~\cite{MD-21}, several new ingredients are required. First, the analysis must be adapted to the present causal, non-exchangeable decoder dynamics, where the token label persists in the limit as the macroscopic variable $\upsigma$ and where the dependency graph is triangular rather than symmetric.
Second, the interaction scale at token $j$ is of order $j^{-1}$ rather than $N^{-1}$, leading to a singular cumulative influence of early indices on later ones; see Section~\ref{sec:app}.
This disrupts standard propagation-of-chaos mechanisms and gives rise to nonstandard correction terms in the correlation estimates; see Lemma~\ref{lem:correl}.

\medskip
In the specific case of iid uniformly distributed prompts, thus satisfying~\eqref{eq:init-conv} with $f_\circ\equiv1$, the mean-field solution remains constant $f\equiv1$, the autocorrelation reduces to $A_\upvarphi\equiv\upvarphi-\int_\T\upvarphi$, and the cross-correlation equation~\eqref{eq:Cphi-lambda} becomes
\[
\left\{\begin{array}{l}\displaystyle
\partial_tC_\upvarphi(t,\upsigma,\uptheta;\upsigma_0)
=-\left(\int_{\upsigma_0}^\upsigma \mathsf{k}_{\uplambda}(\upsigma,\upsigma')(\mathsf{w}_{\upbeta}''\ast_\uptheta C_\upvarphi)(t,\upsigma',\uptheta;\upsigma_0)\,\mathrm d\upsigma'
+\mathsf{k}_{\uplambda}(\upsigma,\upsigma_0)(\mathsf{w}_{\upbeta}''\ast_\uptheta\upvarphi)(\uptheta)\right),\\
C_\upvarphi(0,\upsigma,\uptheta;\upsigma_0)=0.
\end{array}\right.
\]
Taking the Fourier transform on $\T$, this linear equation diagonalizes
and we obtain
\[\hat C_\upvarphi(t,\upsigma,n;\upsigma_0)\,=\,\hat\upvarphi(n)\,g_{a_n}(t,\upsigma;\upsigma_0),\]
where
\[a_n:=n^2\widehat{\mathsf{w}}_{\upbeta}(n),\]
and where for $a\in\R$ we define the profile $g_a$ as the solution of the following Volterra--Hardy equation,
\begin{equation}\label{eq:Hardy-hom}
\left\{\begin{array}{l}\displaystyle
\partial_tg_a(t,\upsigma;\upsigma_0)
-a\left(
\mathsf{k}_{\uplambda}(\upsigma,\upsigma_0)+\int_{\upsigma_0}^{\upsigma}\mathsf{k}_{\uplambda}(\upsigma,\upsigma')g_a(t,\upsigma';\upsigma_0)\,\mathrm d\upsigma'
\right)=0,\\
g_a(0,\upsigma;\upsigma_0)=0.
\end{array}\right.
\end{equation}
By Theorem~\ref{th:lost} and the Fourier representation~\eqref{eq:Acc-soft-Fourier}, using this representation for limiting correlations, we deduce the following approximation for the soft accuracy:
\begin{equation}\label{eq:soft-accuracy-expansion}
\mathscr{A}_N(t,\upsigma_0)
=
\frac{\sqrt{\uppi/2}}{M}
+\frac{\sqrt{2\uppi}}{MN}\,\mathscr{S}_t(\upsigma_0)
+O\left(N^{-1-\upzeta}\upsigma_0^{-2}M^Ce^{Ct}\right),
\end{equation}
where the leading correction is given by
\begin{equation}\label{eq:Sb-def}
\mathscr{S}_t(\upsigma_0)
:=
\sum_{n\ge1}e^{-\frac{\uppi^2}{2M^2}n^2}g_{a_n}(t,1;\upsigma_0).
\end{equation}
We show that this quantity is indeed $\mathsf{U}$-shaped, thus yielding a rigorous explanation of the empirically observed lost-in-the-middle phenomenon; see also Figure~\ref{fig:alibi-U-grid}.

\begin{theor}[Lost in the middle]\label{thm:U-shape}
Let $\uplambda>0$ and assume that $t$ is small enough so that\footnote{As $\mathsf{w}_\upbeta$ is smooth, we have $a_n\to0$ as $n\to\infty$, so this condition is non-empty for $t>0$.}
\begin{equation}\label{eq:affine-smallness}
t\sup_{n\ge1} a_n
\le
\min\left\{3-\sqrt3\,,\, 2\left(1-e^{-\uplambda}\right)\right\}.
\end{equation}
Then the following hold.
\begin{enumerate}[(i)]
\item \emph{Primacy:} $\mathscr{S}_t(\upsigma_0)\to+\infty$ as $\upsigma_0\downarrow0$.
\item \emph{Recency:} $\mathscr{S}_t'(1^-)>0$.
\item \emph{$\mathsf U$-shape:} The function $\upsigma_0\mapsto \mathscr{S}_t(\upsigma_0)$ has a unique global minimum in $(0,1)$.
\end{enumerate}
\end{theor}

\begin{figure}[htp!]
\centering
\includegraphics[width=.31\textwidth]{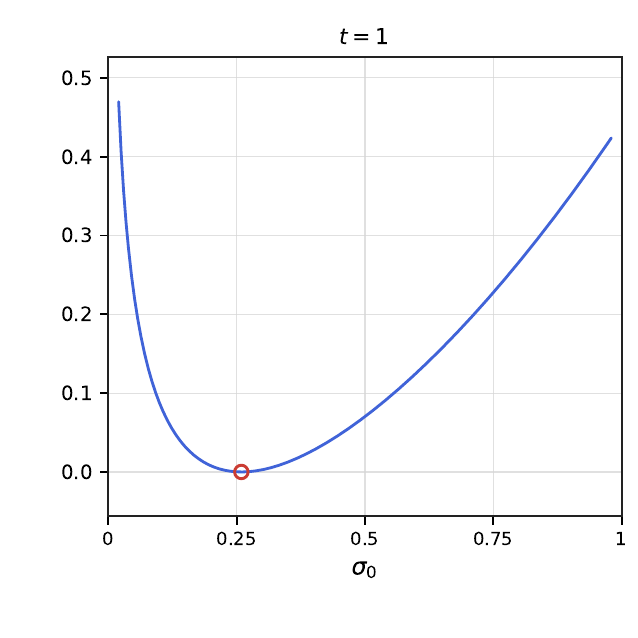}
\includegraphics[width=.31\textwidth]{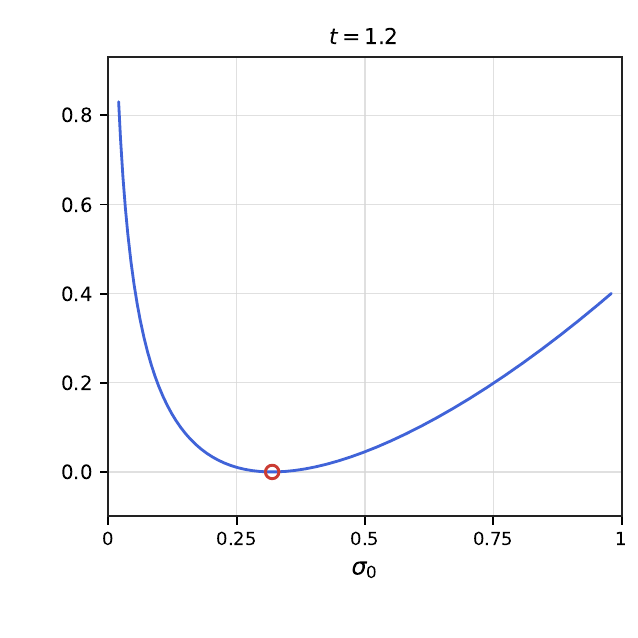}
\includegraphics[width=.31\textwidth]{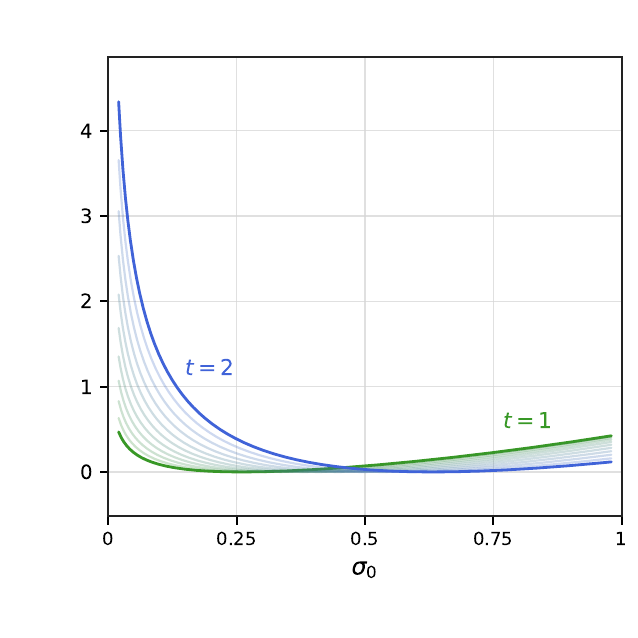}
\vspace{-0.5cm}\caption{The profile predicted by Theorem~\ref{thm:U-shape} in the explicit regime $\upbeta=\uplambda=1$ and $M=8$; the panels the centered correction $\upsigma_0\mapsto \mathscr{S}_t(\upsigma_0)-\min \mathscr{S}_t$ from~\eqref{eq:Sb-def}.}
\label{fig:alibi-U-grid}
\end{figure}

The two sides of the \(\mathsf U\)-shaped profile originate from distinct features of the correlation dynamics. Primacy results from the cumulative causal amplification of early source positions through the triangular interaction structure, whereas recency is produced by the positional bias, which directly favors sources close to the output. Related results on the emergence of primacy and recency biases for Transformers have recently appeared for simplified or discrete-time models, using rather different approaches; see~\cite{wu2025on, herasimchyk2026residual, chowdhury2026lost}. While these two tendencies are already suggested by the causal and positional structure of the model, their quantitative competition is not: in particular, neither mechanism alone predicts the existence, location, or uniqueness of an interior minimum. These features emerge from the leading correlation dynamics, encoded here by the Volterra--Hardy equation~\eqref{eq:Hardy-hom}.

Beyond the specific lost-in-the-middle phenomenon, this illustrates a broader role for kinetic theory in the analysis of Transformer architectures. Mean-field limits describe the evolution of token distributions, but may discard precisely the source-position information relevant to architectural observables such as retrieval. In the present model, this information first appears at the level of $O(N^{-1})$ correlations. Thus, going beyond mean field provides a quantitative description of an architectural effect that is invisible at the level of the limiting token distribution.

\medskip
The proof of Theorem~\ref{thm:U-shape} proceeds by reducing the Volterra--Hardy equation~\eqref{eq:Hardy-hom} to a Goursat problem via a change of variables, which can be solved explicitly in terms of modified Bessel functions. The crucial computation is:

\begin{lem} \label{prop:Bessel}
For $a>0$, $\uplambda\in\R\setminus\{0\}$, and $0<\upsigma_0<\upsigma\le1$, the unique solution of \eqref{eq:Hardy-hom} is given by
\begin{equation}\label{eq:Bessel-full}
g_a(t,\upsigma;\upsigma_0)
=
\frac{\uplambda e^{-\uplambda(\upsigma-\upsigma_0)}}{1-e^{-\uplambda \upsigma}}
\sqrt{\frac{at}{Y(\upsigma;\upsigma_0)}}
\,I_1 \left(2\sqrt{at\,Y(\upsigma;\upsigma_0)}\right),
\end{equation}
where
\begin{equation}\label{eq:Ylambda}
Y(\upsigma;\upsigma_0):=
\log \frac{e^{\uplambda \upsigma}-1}{e^{\uplambda\upsigma_0}-1}.
\end{equation}
For $\uplambda=0$, this reduces to
\begin{equation}\label{eq:Bessel-output}
g_a(t,\upsigma;\upsigma_0)
=
\frac{1}{\upsigma}\sqrt{\frac{at}{\log(\upsigma/\upsigma_0)}}\,
I_1 \left(2\sqrt{at\,\log(\upsigma/\upsigma_0)}\right).
\end{equation}
\end{lem}

\subsection{Extensions}
Our results can be extended in several natural directions.
We first briefly discuss the extension of Theorem~\ref{thm:U-shape} beyond uniformly distributed initial data, focusing on the near-deterministic regime. 
We then turn to several extensions of the model itself. The specific choice of interaction kernel $\mathsf{w}_{\upbeta}(\uptheta)=e^{\upbeta\cos\uptheta}$ plays no essential role in the mean-field and correlation analysis: it can be replaced by any smooth even periodic interaction kernel. Likewise, the arguments extend to higher-dimensional phase spaces.
Finally, we discuss other choices of positional encodings, the extension to the more general causal self-attention dynamics, and the connection between our leading-order correlation analysis and a broader fluctuation theory for the empirical measure.

\subsubsection{Near-deterministic initial data}
Beyond the specific case of uniformly distributed initial data, we show that Theorem~\ref{thm:U-shape} also holds for near-deterministic initial data. More precisely, consider independent perturbations of a deterministic profile,
\[\uptheta_j^\upvarepsilon(0)=h_\circ(j/N)+\upvarepsilon\upxi_j\pmod{2\uppi},\qquad
\upxi_j\sim\upnu_{j/N},\]
where the centered laws $\upnu_\upsigma$ have variance $v(\upsigma)$. For $\upvarepsilon=0$, the mean-field solution is monokinetic, that is, deterministic conditional on the positional label,
\[f^0(t,\upsigma)=\updelta_{h(t,\upsigma)},\]
while all statistical correlations vanish.  The first nontrivial correction is of order~$O(\upvarepsilon^2)$. We show that the rescaled cross-correlations converge, as $N\uparrow\infty$ and then $\upvarepsilon\downarrow0$, to a linear response of the form
\[-v(\upsigma_0)\upvarphi'(h_\circ(\upsigma_0))R(t,\upsigma;\upsigma_0)\,\updelta'_{h(t,\upsigma)},\]
where $R$ solves a Volterra equation; see Proposition~\ref{prop:near-monokinetic}.  When $h_\circ$ is constant, this response equation reduces to the Volterra--Hardy equation \eqref{eq:Hardy-hom}.  Consequently, the leading small-noise correction to the soft accuracy has the same $\mathsf U$-shaped profile as in Theorem~\ref{thm:U-shape}; see Corollary~\ref{cor:homogeneous-cluster}.

\subsubsection{On positional encodings}\label{sssec:positional-encodings}
We focus on the ALiBi encoding factor for simplicity, but the proofs of Theorems~\ref{th:mf-gpt} and~\ref{th:lost} extend with only minor modifications to more general causal weights of the form
\[b(\tfrac{j-k}N)\mathds1_{k<j},\qquad b:[0,1]\to(0,\infty).\]
For continuous $b$, the qualitative mean-field limit of Theorem~\ref{th:mf-gpt}(i) carries over with the limiting graphon $\mathsf{k}_{\uplambda}(\upsigma,\upsigma_0)$ replaced by
\[\mathsf{k}_{b}(\upsigma,\upsigma_0):=\frac{b(\upsigma-\upsigma_0)}{\int_0^\upsigma b(\upsigma-\upsigma')\mathrm d\upsigma'}\mathds1_{\upsigma_0<\upsigma}.\]
For $b\in C^1$, the estimates in Theorems~\ref{th:mf-gpt}(ii) and~\ref{th:lost} adapt with only
minor modifications.

In the iid uniform regime, each Fourier mode of the limiting cross-correlation still satisfies the Volterra equation~\eqref{eq:Hardy-hom}, with $\mathsf{k}_{\uplambda}$ replaced by $\mathsf{k}_b$, and the causal primacy--recency mechanism of Theorem~\ref{thm:U-shape} is therefore easily adapted. Primacy follows from the same causal amplification, while recency is controlled locally by the behavior of $b$ at the origin. More precisely, the condition $b'(0)<0$ yields recency under a suitable small-time condition. By contrast, the full $\mathsf U$-shape statement, namely the uniqueness of the interior minimum, is a global property of the positional profile and would require additional structural assumptions on $b$.

Avoiding the $\mathsf U$-shaped mechanism therefore naturally suggests weakening the monotone recency bias or introducing a competing phase-sensitive positional effect, as in empirical approaches~\cite{jiang-etal-2024-longllmlingua, hsieh-etal-2024-found}. RoPE-type encodings~\cite{Su2021RoFormer} provide a natural example, but lie outside the present theory because the positional dependence then enters the interaction phase rather than through a scalar causal kernel. We leave such extensions for future work.

\subsubsection{General self-attention model}

In view of~\cite{dey2026don, KoubbiGeshkovskiRigollet2026}, the dynamics~\eqref{eq:gpt-dyn-alibi} can be viewed as the $d=2$, $\uprho(t) \equiv \updelta_{\Id, \Id,\Id}$ specialization of the causal self-attention dynamics with ALiBi encoding in the complete depth scaling:
\begin{equation}\label{eq:full-dyn}
\frac{\mathrm d}{\mathrm dt} x_j(t)
=
\frac{1}{\mathscr{Z}_{N,j}(t)}
\mathbb{E}_{(Q,K,V)\sim\uprho(t)}\sum_{k=1}^{j-1}
\exp\left(
\upbeta \left\langle Qx_j(t), Kx_k(t)\right\rangle
-\frac{\uplambda}{N}(j-k)
\right)
\mathsf{P}_{x_j(t)}^\perp Vx_k(t),
\end{equation}
where $x_j(t)\in\mathbb S^{d-1}$,
$\mathsf{P}_x^\perp y:=y-\left\langle x,y\right\rangle x$,
and 
\[
\mathscr{Z}_{N,j}(t)
:=
\mathbb{E}_{(Q,K,V)\sim\uprho(t)}\sum_{k=1}^{j-1}
\exp\left(
\upbeta \left\langle Qx_j(t), Kx_k(t)\right\rangle
-\frac{\uplambda}{N}(j-k)
\right).
\]
Here $\uprho(t)$ describes the layerwise distribution of query, key, and value matrices, which is obtained by training on data. Compared with the reduced model~\eqref{eq:gpt-dyn-alibi}, the normalization now also depends on the current configuration through the attention scores.
A full analysis of~\eqref{eq:full-dyn} is left for future work. Nevertheless, the numerical experiments in Section~\ref{sec:highdim-numerics} indicate that the $\mathsf U$-shaped retrieval profile persists for random near-identity matrices in dimensions $d>2$.

\subsubsection{Fluctuation interpretation}\label{sec:gen-fluct}
The correlation equations obtained in Theorem~\ref{th:lost} admit a natural interpretation in terms of fluctuations of the empirical measure. Indeed, under suitable assumptions on the initial fluctuations, one expects the empirical fluctuation field
\[\upzeta_N(t):=\sqrt N(\upmu_N(t)-f(t))\]
to converge in law to a random distribution $\upzeta(t)$ solving the linearization of the mean-field equation~\eqref{eq:mfl-lambda} around $f$, with initial data given by the limiting initial fluctuations. Theorem~\ref{th:lost} already identifies the corresponding two-time covariance kernel in the case of independent initial data: setting
\[H_\upvarphi(t,\upsigma,\uptheta;\upsigma_0)=A_\upvarphi(t,\upsigma_0,\uptheta)\updelta_{\upsigma_0}(\upsigma)+\mathbf1_{\upsigma>\upsigma_0}C_\upvarphi(t,\upsigma,\uptheta;\upsigma_0),\]
the equations~\eqref{eq:Aphi-lambda}--\eqref{eq:Cphi-lambda} combine into the single linearized mean-field equation
\begin{equation*}
\left\{\begin{array}{l}
\displaystyle\partial_tH_\upvarphi(t,\upsigma,\uptheta;\upsigma_0)+\partial_\uptheta\left(H_\upvarphi(t,\upsigma,\uptheta;\upsigma_0)\int_0^\upsigma \mathsf{k}_\uplambda(\upsigma,\upsigma')(\mathsf{w}_\upbeta'\ast_\uptheta f)(t,\upsigma',\uptheta)\,\mathrm d\upsigma'\right)\\
\hspace{4cm}\displaystyle+\partial_\uptheta\left(f(t,\upsigma,\uptheta)\int_{\upsigma_0}^\upsigma\mathsf{k}_\uplambda(\upsigma,\upsigma')(\mathsf{w}_\upbeta'\ast_\uptheta H_\upvarphi)(t,\upsigma',\uptheta;\upsigma_0)\,\mathrm d\upsigma'\right)=0,\\
H_\upvarphi(0,\upsigma,\uptheta;\upsigma_0)=\updelta_{\upsigma_0}(\upsigma)A_\upvarphi(0,\upsigma,\uptheta).
\end{array}\right.
\end{equation*}
Thus $A_\upvarphi$ and $C_\upvarphi$ are respectively the diagonal and off-diagonal components of the covariance kernel associated with the linearized fluctuation dynamics. 
In the iid uniform regime, the Volterra--Hardy kernel $g_a$ can then be viewed as the integral kernel of the corresponding fluctuation semigroup.
A systematic fluctuation theory, allowing in particular for general initial fluctuations, is postponed to a future work.

\subsection{Organization}

The paper is organized as follows. Section~\ref{sec:app} records some elementary convergence estimates for interaction weights.
Section~\ref{sec:cumulants} develops the cumulant estimates that form the backbone of all our convergence proofs: using Glauber calculus adapted to the triangular causal structure of the dynamics, we derive sharp bounds on covariances and third cumulants of particle trajectories. Section~\ref{sec:mfl-proof} contains the proof of Theorem~\ref{th:mf-gpt} on the mean-field limit. Section~\ref{sec:corr-proof} gives the proof of Theorem~\ref{th:lost} on the characterization of correlations, while Section~\ref{sec:litm} completes the proof of Theorem~\ref{thm:U-shape}. Finally, Section~\ref{sec:extensions} provides numerical evidence supporting the robustness of the lost-in-the-middle phenomenon and extends Theorem~\ref{thm:U-shape} to perturbations of deterministic data.

\section{Interaction weights and limiting graphon}\label{sec:app}

We introduce the following short-hand notation for the non-exchangeable, directed interaction weights appearing in the model~\eqref{eq:gpt-dyn-alibi}: for~$1\le k,j\le N$, set
\begin{equation}\label{eq:shortdef-omega}
\upomega_{j,k}
:=
\frac{e^{-\frac{\uplambda}{N}(j-k)}\mathds1_{k<j}}{\mathscr{Z}_{N,j}},
\end{equation}
where the normalization $\mathscr{Z}_{N,j}$ ensures
$\sum_{k=1}^N\upomega_{j,k}=1$ for $1< j\le N$.
A direct computation yields, uniformly for all $1\le k,j\le N$,
\[
\upomega_{j,k}
=
\frac{e^{\frac{\uplambda}{N}k}\mathds1_{ k<j}}{\sum_{m=1}^{j-1}e^{\frac{\uplambda}{N}m}}
\simeq j^{-1}\mathds1_{k<j}.
\]
In particular, this implies that the weights are row-stochastic but not column-balanced, as already emphasized in~\eqref{eq:min-requ-jabinsoler-0}:
\begin{equation}\label{eq:min-requ-jabinsoler}
\sup_{1\le k\le N}\sum_{j=1}^N\upomega_{j,k}\simeq\log N.
\end{equation}
This logarithmic divergence reflects the cumulative influence of small indices on later particles and places the model outside the standard framework for mean-field/graph limits, namely beyond the standard boundedness requirement of e.g.~\cite[Assumption~(3)]{JabinPoyatoSoler2025}.

We next introduce a continuum representation of this interaction structure. Define the rescaled interpolated kernel
\begin{equation}\label{eq:def-KN}
\mathsf{K}_N(\upsigma,\upsigma')
:=N\upomega_{\lceil N\upsigma\rceil,\lceil N\upsigma'\rceil},\qquad \upsigma,\upsigma'\in(0,1].
\end{equation}
This can be viewed as a continuum interaction kernel associated with the underlying directed weighted graph (a directed graphon in the sense of graph limit theory).
By construction, it satisfies
\[\int_0^1 \mathsf{K}_N(\upsigma,\upsigma')\,\mathrm d\upsigma'=1,\qquad\upsigma\in\left(N^{-1},1\right],\]
and the pointwise uniform bound
\[
\mathsf{K}_N(\upsigma,\upsigma')\lesssim \upsigma^{-1}\mathds1_{\upsigma'<\upsigma}.
\]
It is easily checked that $\mathsf{K}_N$ converges almost everywhere to the limiting kernel~$\mathsf{k}_{\uplambda}$ defined in~\eqref{eq:kernel-Klambda}. The latter satisfies the same normalization and pointwise bound,
\[
\int_0^1\mathsf{k}_\uplambda(\upsigma,\upsigma')\,\mathrm d\upsigma'=1,\qquad\mathsf{k}_{\uplambda}(\upsigma,\upsigma')\lesssim \upsigma^{-1}\mathds1_{\upsigma'<\upsigma},
\]
while the logarithmic divergence in~\eqref{eq:min-requ-jabinsoler} is reflected by the singular behavior
\[\int_{0}^1 \mathsf{k}_{\uplambda}(\upsigma,\upsigma')\,\mathrm d\upsigma
\sim |\!\log \upsigma'|,
\qquad \upsigma'\downarrow0,\]
which will create technical difficulties in the analysis.
Finally, we quantify the convergence $\mathsf{K}_N\to\mathsf{k}_{\uplambda}$.
For $\uplambda\neq0$, using the geometric-series identity
\[
\mathscr{Z}_{N,j}
=
\frac{1-e^{-\frac{\uplambda}{N}(j-1)}}{e^{\frac{\uplambda}{N}}-1},
\]
we obtain, uniformly for $1\le k<j\le N$,
\[
N\upomega_{j,k}
=
\frac{N(e^{\frac{\uplambda}{N}}-1)}{1-e^{-\frac{\uplambda}{N}(j-1)}}\,
e^{-\frac{\uplambda}{N}(j-k)}
=
\mathsf{k}_{\uplambda}\left(\frac{j}{N},\frac{k}{N}\right)
+O\left(\frac{1}{N(\frac jN)^2}\right).
\]
For $\uplambda=0$, the same estimate is immediate from
\[
N\upomega_{j,k}=\frac{N}{j-1}=\frac{1}{j/N}+O\left(\frac{1}{N(\frac jN)^2}\right).
\]
Consequently, we can deduce
for all $\upsigma\in(0,1]$,
\begin{equation}\label{eq:estim-KNk}
\int_0^1|\mathsf{K}_N(\upsigma,\cdot)-\mathsf{k}_{\uplambda}(\upsigma,\cdot)|\le\frac C{N\upsigma+1}.
\end{equation}
This convergence estimate will be used repeatedly in the sequel.

\section{Cumulant bounds}\label{sec:cumulants}

The mean-field limit and the characterization of correlations both rely on quantitative control of the statistical dependence between particle trajectories. Provided that the initial data are independent, we can use the so-called Glauber calculus of~\cite{MD-21} to estimate correlations. However, compared with~\cite{MD-21}, some important care is needed in the present non-exchangeable setting: the smallness of the interaction at particle $j$ is of order $j^{-1}$ rather than $N^{-1}$, which creates difficulties when $j$ is small, as reflected for instance by~\eqref{eq:min-requ-jabinsoler}. For this reason, nontrivial corrections appear in correlation estimates, see~\eqref{eq:est-cov-th} and~\eqref{eq:est-kap3-th} below,
but they are harmless when applied to macroscopic indices $j=\lceil N\upsigma\rceil$ with $\upsigma>0$ bounded away from zero.

We start by recalling tools from Glauber calculus with respect to independent initial data.
Let $\upnu_j:=\mathrm{Law}(\uptheta_j^\circ)$ and let 
\(
(\Upomega,\mathcal F,\Pp):=\bigotimes_{j=1}^N\left(\T,\mathcal B(\T),\upnu_j\right)
\)
be the product probability space carrying the independent initial data. We
identify $\uptheta_j^\circ$ with the $j$-th coordinate map on $\Upomega$. For
each $k$, we denote by
\[
\E_k^\circ[Y]:=\E\left[Y\,\middle|\,(\uptheta_j^\circ)_{j:j\ne k}\right]
\]
the conditional expectation given all variables except $\uptheta_k^\circ$, or equivalently the integration with respect to $\uptheta_k^\circ$, and we define the Glauber derivative by
\[
D_k^\circ Y:=Y-\E_k^\circ[Y].
\]
The associated Glauber Laplacian is
\[
\mathcal L^\circ:=\sum_{k=1}^N (D_k^\circ)^*D_k^\circ=\sum_{k=1}^N D_k^\circ,
\]
and we write $\mathcal T^\circ:=(\mathcal L^\circ)^{-1}$ for its inverse defined on centered random variables. We shall only use the following standard properties, proved e.g.\@ in~\cite[Lemmas~2.4(v) and~2.5]{MD-21}: for centered random variables~$Y,Y'\in L^2(\Upomega)$ and all $1<p<\infty$,
\begin{equation}\label{eq:glauber-rep-covariance}
\Cov\left(Y,Y'\right)=\sum_{k=1}^N\E\left[(D_k^\circ Y)\,\mathcal T^\circ (D_k^\circ Y')\right],
\qquad
\left\|\mathcal T^\circ(D_k^\circ Y)\right\|_{L^p(\Upomega)}\lesssim \left\|D_k^\circ Y\right\|_{L^p(\Upomega)},
\end{equation}
hence
\begin{equation}\label{eq:glauber-covariance}
\left|\Cov\left(Y,Y'\right)\right|
\lesssim
\sum_{k=1}^N\left\|D_k^\circ Y\right\|_{L^2(\Upomega)}\left\|D_k^\circ Y'\right\|_{L^2(\Upomega)}.
\end{equation}

Our main correlation estimates take on the following guise. For our purposes, it is enough to focus on second and third cumulants, but the proof extends easily to higher cumulants. Recall that the third joint cumulant is defined as $\upkappa_{1,1,1}(Y,Y',Y'')=\E[(Y-\E[Y])(Y'-\E[Y'])(Y''-\E[Y''])]$.

\begin{lem}\label{lem:correl}
Assume that the initial data $(\uptheta_j^0)_{1\leq j \leq N}$ are independent. Then, for all $\upvarphi\in C^2(\T)$, $t,t',t''\in[0,T]$, and all~$i>j>k$,
\begin{eqnarray}
\left|\Cov\left(\upvarphi\left(\uptheta_i(t)\right),\upvarphi\left(\uptheta_j(t')\right)\right)\right|&\lesssim&
i^{-1}e^{\sqrt{C T\log(i/j)}}e^{C T}\|\upvarphi'\|_{L^\infty(\T)}^2,
\label{eq:est-cov-th}\\
\hspace{-1cm}\left|\upkappa_{1,1,1}\left(\upvarphi\left(\uptheta_i(t)\right),\upvarphi\left(\uptheta_j(t')\right),\upvarphi\left(\uptheta_k(t'')\right)\right)\right|&\lesssim&(ij)^{-1}e^{\sqrt{C T\log(i/k)}}e^{C T}\|\upvarphi'\|_{L^{\infty}(\T)}^2\|\upvarphi'\|_{W^{1,\infty}(\T)}.\label{eq:est-kap3-th}
\end{eqnarray}
\end{lem}

\begin{proof}
We split the proof into five steps: we start by estimating the sensitivity of trajectories with respect to initial data in form of uniform bounds on Glauber derivatives, and we then conclude by reconstructing cumulants as sums of products of Glauber derivatives. Compared to the covariance, the third cumulant further requires bounds on second Glauber derivatives, which is postponed to the last two steps of the proof. 

\medskip
\step1 First Glauber derivatives: proof that for all $1\le i,k\le N$ and $t\ge0$, almost surely,
\begin{equation}\label{eq:est-Gl}
\left|D_k^\circ \uptheta_i(t)\right|\le C e^{C t}\mathds1_{k=i}+C i^{-1}e^{\sqrt{C t\log(i/k)}}e^{C t}\mathds1_{k<i}.
\end{equation}
Note that for a Lipschitz function $h$ we can bound
\begin{equation}
\left|D_k^\circ h\left(Y\right)\right|\le\|h'\|_{L^\infty(\T)}\left(\left|D_k^\circ Y\right|+\E\left[\left| D_k^\circ Y\right|\,\middle|\,(\uptheta_j^\circ)_{j:j\ne k}\right]\right),
\end{equation}
and thus, for any $p\ge1$,
\begin{equation}\label{eq:chain-rule-Glaub}
\|D_k^\circ h(Y)\|_{L^p(\Upomega)}\le2\|h'\|_{L^\infty(\T)}\|D_k^\circ Y\|_{L^p(\Upomega)}.
\end{equation}
Taking Glauber derivatives in the particle dynamics and using this latter estimate, we find
\[\frac{\mathrm d}{\mathrm dt}\|D_k^\circ\uptheta_i\|_{L^p(\Upomega)}\le2\|\mathsf{w}_{\upbeta}''\|_{L^\infty(\T)}\|D_k^\circ\uptheta_i\|_{L^p(\Upomega)}+2\|\mathsf{w}_{\upbeta}''\|_{L^\infty(\T)}\sum_{j=1}^{i-1}\upomega_{i,j}\|D_k^\circ\uptheta_j\|_{L^p(\Upomega)}.\]
Integrating in time and recalling $\upomega_{i,j}\lesssim i^{-1}$, this yields
\[\|D_k^\circ\uptheta_i(t)\|_{L^p(\Upomega)}\le e^{C t}\|D_k^\circ\uptheta_i^\circ\|_{L^p(\Upomega)}+ i^{-1}\sum_{j=1}^{i-1} \int_0^t C e^{C(t-s)} \|D_k^\circ\uptheta_j(s)\|_{L^p(\Upomega)}\,\mathrm d s.\]
As initial data are independent, we note that $D_k^\circ\uptheta_i^\circ=0$ for $k\ne i$, and the triangular structure of the dynamics further ensures $D_k^\circ\uptheta_i(t)=0$ for all $k>i$ and $t\ge0$. The above then becomes
\[\|D_k^\circ\uptheta_i(t)\|_{L^p(\Upomega)}\le Ce^{C t}\mathds1_{k=i}+ i^{-1}\sum_{j=k}^{i-1} \int_0^t C e^{C(t-s)} \|D_k^\circ\uptheta_j(s)\|_{L^p(\Upomega)}\,\mathrm d s.\]
Iterating this estimate, we get for $k<i$,
\[\|D_k^\circ\uptheta_i(t)\|_{L^p(\Upomega)}\le i^{-1} C e^{C t}\sum_{r=0}^{i-k-1}\frac{(C t)^{r}}{r!}\sum_{k<j_r<\ldots<j_1<i}(j_1\ldots j_r)^{-1},\]
and thus, after straightforward calculation,
\[\|D_k^\circ\uptheta_i(t)\|_{L^p(\Upomega)}\le i^{-1} C e^{C t}\sum_{r=0}^{i-k-1}\frac{(C t\log(i/k))^r}{r!^2}\le i^{-1} C e^{C t}e^{\sqrt{C t\log(i/k)}},\]
that is, \eqref{eq:est-Gl}.

\medskip
\step2 Proof of~\eqref{eq:est-cov-th}.\\
Appealing to the Glauber covariance estimate~\eqref{eq:glauber-covariance}, applying it to trajectories, and recalling the chain rule~\eqref{eq:chain-rule-Glaub}, we get
\begin{equation}\label{eq:est-cov-phi}
\left|\Cov\left(\upvarphi\left(\uptheta_i(t)\right),\upvarphi\left(\uptheta_j(t')\right)\right)\right|\,\lesssim\,\|\upvarphi'\|_{L^\infty(\T)}^2\sum_{k=1}^N\left\|D_k^\circ \uptheta_i(t)\right\|_{L^2(\Upomega)}\left\|D_k^\circ \uptheta_j(t')\right\|_{L^2(\Upomega)}.
\end{equation}
Inserting~\eqref{eq:est-Gl}, for $i> j$, we obtain
\begin{multline*}
\left|\Cov\left(\upvarphi\left(\uptheta_i(t)\right),\upvarphi\left(\uptheta_j(t')\right)\right)\right|\,\lesssim\,
i^{-1}e^{\sqrt{C t\log(i/j)}}e^{C(t+ t')}\|\upvarphi'\|_{L^\infty(\T)}^2\\
+(ij)^{-1}e^{C(t+t')}\|\upvarphi'\|_{L^\infty(\T)}^2\sum_{k=1}^{j-1}e^{\sqrt{C t\log(i/k)}}e^{\sqrt{C t'\log(j/k)}},
\end{multline*}
and the claim~\eqref{eq:est-cov-th} follows after a straightforward estimate of the last sum.

\medskip
\step3
Second Glauber derivatives: proof that for all $1\le k<j<i\le N$ and $t\ge0$,
\begin{equation}\label{eq:est-Glaub-2nd}
|D_j^\circ D_k^\circ\uptheta_i(t)|\,\le\,C (ij)^{-1} e^{\sqrt{C t\log(i/k)}}e^{C t},
\end{equation}
Arguing similarly as for~\eqref{eq:chain-rule-Glaub}, we find the following second-order chain rule: for a smooth function $h$, for all $j\ne k$ and $2\le p\le\infty$,
\begin{equation}\label{eq:chain-rule-Glaub-2nd}
\|D_j^\circ D_k^\circ h(Y)\|_{L^p(\Upomega)}\le 2\|h'\|_{L^\infty(\T)}\|D_j^\circ D_k^\circ Y\|_{L^p(\Upomega)}+4\|h''\|_{L^\infty(\T)}\|D_j^\circ Y\|_{L^p(\Upomega)}\|D_k^\circ Y\|_{L^p(\Upomega)}.
\end{equation}
For $k<j<i$, taking second Glauber derivatives in the particle dynamics, using this estimate, and recalling $\upomega_{i,j}\lesssim i^{-1}$, we find
\begin{multline*}
\frac{\mathrm d}{\mathrm dt}\|D_j^\circ D_k^\circ\uptheta_i \|_{L^p(\Upomega)}
\le C \|D_j^\circ D_k^\circ \uptheta_i\|_{L^p(\Upomega)}
+C i^{-1}\sum_{l=1}^{i-1}\|D_j^\circ D_k^\circ\uptheta_l\|_{L^p(\Upomega)}\\
+C i^{-1}\sum_{l=1}^{i-1}\|D_j^\circ (\uptheta_i-\uptheta_l)\|_{L^p(\Upomega)}\|D_k^\circ (\uptheta_i-\uptheta_l)\|_{L^p(\Upomega)}.
\end{multline*}
Using~\eqref{eq:est-Gl} to estimate the last sum, we get
\begin{multline*}
\frac{\mathrm d}{\mathrm dt}\|D_j^\circ D_k^\circ\uptheta_i \|_{L^p(\Upomega)}
\le C \|D_j^\circ D_k^\circ \uptheta_i\|_{L^p(\Upomega)}
+C i^{-1}\sum_{l=1}^{i-1}\|D_j^\circ D_k^\circ\uptheta_l\|_{L^p(\Upomega)}\\
+C i^{-2}e^{\sqrt{Ct\log(i/k)}}e^{C t}
+C (ij)^{-1}e^{\sqrt{Ct\log(j/k)}}e^{C t}
\end{multline*}
and thus, after time integration,
\begin{multline*}
\|D_j^\circ D_k^\circ\uptheta_i(t) \|_{L^p(\Upomega)}
\le
C i^{-2}e^{\sqrt{Ct\log(i/k)}}e^{C t}
+C (ij)^{-1} e^{\sqrt{Ct\log(j/k)}}e^{C t}\\
+i^{-1}\sum_{l=1}^{i-1}\int_0^t C e^{C (t-s)}\|D_j^\circ D_k^\circ\uptheta_l(s)\|_{L^p(\Upomega)}\,\mathrm d s.
\end{multline*}
Note that the first term is bounded by the second one by a simple monotonicity argument.
Further using~\eqref{eq:est-Gl} in the form
\[\|D_j^\circ D_k^\circ\uptheta_l(t)\|_{L^p(\Upomega)}\le\left\{\begin{array}{lll}
j^{-1}e^{\sqrt{C t\log(j/k)}}e^{C t}&:&l=j> k,\\
0&:&l=k< j,\\
\end{array}\right.\]
the above can be reorganized more simply as
\begin{equation*}
\|D_j^\circ D_k^\circ\uptheta_i(t) \|_{L^p(\Upomega)}
\le
C (ij)^{-1} e^{\sqrt{Ct\log(j/k)}}e^{C t}
+i^{-1}\sum_{j< l<i}\int_0^t C e^{C (t-s)}\|D_j^\circ D_k^\circ\uptheta_l(s)\|_{L^p(\Upomega)}\,\mathrm d s.
\end{equation*}
Now iterating this estimate, and evaluating the resulting series similarly as in Step~1, we find
\begin{equation*}
\|D_j^\circ D_k^\circ\uptheta_i(t) \|_{L^p(\Upomega)}
\le
C (ij)^{-1} e^{\sqrt{C t\log(j/k)}} e^{\sqrt{C t\log(i/j)}}e^{C t},
\end{equation*}
which is equivalent to the claim~\eqref{eq:est-Glaub-2nd}.

\medskip
\step4 Third-order Poincar\'e inequality for third joint cumulants: proof that for all random variables~$Y,Y',Y''\in L^3(\Upomega)$ we have
\begin{multline}\label{eq:estim-kappa3-Gl}
|\upkappa_{1,1,1}[Y,Y',Y'']|
\,\le\,\sum_{j,k=1}^N \left(\|D_j^\circ D_k^\circ Y\|_{L^3(\Upomega)}\|D_j^\circ Y'\|_{L^3(\Upomega)}\|D_k^\circ Y''\|_{L^3(\Upomega)}\right.\\[-4mm]
\left.+\|D_j^\circ Y\|_{L^3(\Upomega)}\|D_j^\circ D_k^\circ Y'\|_{L^3(\Upomega)}\|D_k^\circ Y''\|_{L^3(\Upomega)}\right.\\
\left.+\|D_j^\circ Y\|_{L^3(\Upomega)}\|D_k^\circ Y'\|_{L^3(\Upomega)}\|D_j^\circ D_k^\circ Y''\|_{L^3(\Upomega)}\right).
\end{multline}
Although such Glauber estimates for {\it joint} cumulants were not covered in~\cite{MD-21}, they are easily deduced with the same toolkit developed in~\cite{MD-21}, as we briefly explain.
Without loss of generality, we may assume that $Y,Y',Y''$ have vanishing expectation. Using~\eqref{eq:glauber-rep-covariance},
we can write
\begin{equation*}
\upkappa_{1,1,1}[Y,Y',Y'']
\,=\,\E\left[YY'Y''\right]
\,=\,\Cov(YY',Y'')
\,=\,\sum_{j=1}^N\E\left[D_j^\circ(YY')\mathcal T^\circ(D_j^\circ Y'')\right].
\end{equation*}
Now note that we have the approximate chain rule
\begin{equation*}
\left|D_j^\circ(YY')
-Y(D_j^\circ Y')
-Y'(D_j^\circ Y)\right|
\,\le\,|D_j^\circ Y||D_j^\circ Y'|+\E\left[|D_j^\circ Y||D_j^\circ Y'|\,\middle|\,(\uptheta_l^\circ)_{l:l\ne j}\right].
\end{equation*}
Inserting this into the above and recalling the boundedness of the inverse Glauber Laplacian, cf.~\eqref{eq:glauber-rep-covariance},
we are led to
\begin{multline}\label{eq:est-kappa3-pre}
|\upkappa_{1,1,1}[Y,Y',Y'']|
\,\le\,\sum_{j=1}^N\left(\left|\E\left[Y(D_j^\circ Y')\mathcal T^\circ(D_j^\circ Y'')\right]\right|+\left|\E\left[Y'(D_j^\circ Y)\mathcal T^\circ(D_j^\circ Y'')\right]\right|\right)\\[-3mm]
+\sum_{j=1}^N\|D_j^\circ Y\|_{L^3(\Upomega)}\|D_j^\circ Y'\|_{L^3(\Upomega)}\|D_j^\circ Y''\|_{L^3(\Upomega)}.
\end{multline}
It remains to estimate the first two terms.
Using~\eqref{eq:glauber-rep-covariance} again,
we can write
\begin{equation*}
\E\left[Y(D_j^\circ Y')\mathcal T^\circ(D_j^\circ Y'')\right]
\,=\,\Cov\left(Y,(D_j^\circ Y')\mathcal T^\circ(D_j^\circ Y'')\right)
\,=\,\sum_{k=1}^N \E\left[(D_k^\circ Y)\mathcal T^\circ D_k^\circ\left((D_j^\circ Y')\mathcal T^\circ(D_j^\circ Y'')\right)\right],
\end{equation*}
and thus, by the approximate chain rule and the boundedness of the inverse Glauber Laplacian,
\begin{multline*}
\left|\E\left[Y(D_j^\circ Y')\mathcal T^\circ(D_j^\circ Y'')\right]\right|
\,\lesssim\,\sum_{k=1}^N \left(\|D_k^\circ Y\|_{L^3(\Upomega)}\|D_k^\circ D_j^\circ Y'\|_{L^3(\Upomega)}\|D_j^\circ Y''\|_{L^3(\Upomega)}\right.\\[-3mm]
\left.+\|D_k^\circ Y\|_{L^3(\Upomega)}\|D_j^\circ Y'\|_{L^3(\Upomega)}\|D_k^\circ D_j^\circ Y''\|_{L^3(\Upomega)}\right).
\end{multline*}
Inserting this into~\eqref{eq:est-kappa3-pre}, and noting that for $j=k$ we have $D_j^\circ D_k^\circ=D_j^\circ$, the claim~\eqref{eq:estim-kappa3-Gl} follows.

\medskip
\step5 Proof of~\eqref{eq:est-kap3-th}.\\
Applying~\eqref{eq:estim-kappa3-Gl} to trajectories, recalling the chain rules~\eqref{eq:chain-rule-Glaub} and~\eqref{eq:chain-rule-Glaub-2nd}, and inserting the Glauber estimates~\eqref{eq:est-Gl} and~\eqref{eq:est-Glaub-2nd}, we get for all $k<j<i$ and $t,t',t''\ge0$,
\begin{multline*}
|\upkappa_{1,1,1}[\upvarphi(\uptheta_i(t)),\upvarphi(\uptheta_j(t')),\upvarphi(\uptheta_k(t''))]|
\,\le\,
\|\upvarphi'\|_{L^\infty(\T)}^2\|\upvarphi'\|_{W^{1,\infty}(\T)}(ij)^{-1}e^{\sqrt{C T\log(i/k)}}e^{C T}\\
\times\left(1
+\sum_{l\,:\,k<l<j}l^{-1}
+k^{-1}\sum_{l\,:\,l<k}e^{\sqrt{C T\log(k/l)}}
+k^{-1}\sum_{l\,:\,l<j}l^{-1}\sum_{m\,:\,m<k\wedge l}e^{\sqrt{C T\log(k/m)}}
\right)
\end{multline*}
where we have set for abbreviation $T:=t+t'+t''$. The sums in the last parenthesis can be bounded by $\log(j/k)\le\log(i/k)$ and the claim~\eqref{eq:est-kap3-th} follows.
\end{proof}

\section{Mean-field limit}\label{sec:mfl-proof}

This section proves Theorem~\ref{th:mf-gpt}.
We split the argument into four steps. The qualitative mean-field limit is proven in Step~1 and follows from standard compactness arguments together with a uniqueness result for the mean-field equation~\eqref{eq:mfl-lambda}. The error estimates are more delicate and occupy the remaining three steps of the proof. They are based on the covariance estimate of Section~\ref{sec:cumulants} and on some stability analysis for the mean-field equation.

Next to the empirical measure~\eqref{eq:emp-meas}, we shall also need to work with the marginal distribution of the individual particles:
for $t\ge0$ and $\upsigma\in(0,1]$, let $f_N(t,\upsigma,\cdot)\in \Pc(\T)$ denote the probability density of $\uptheta_{\lceil N\upsigma\rceil}(t)$,
\begin{equation}\label{eq:marg-fN}
\int_\T\uppsi(\uptheta) f_N(t,\upsigma,\mathrm d\uptheta) =\E\left[\uppsi(\uptheta_{\lceil N\upsigma\rceil}(t))\right],\qquad\uppsi\in C^\infty(\T).
\end{equation}

\begin{proof}[Proof of Theorem~\ref{th:mf-gpt}]

We proceed in four steps.

\medskip
\step1~Qualitative mean-field limit.\\
In terms of the empirical measure~\eqref{eq:emp-meas} and the kernel $\mathsf{K}_N$ defined in~\eqref{eq:def-KN}, the particle dynamics~\eqref{eq:gpt-dyn-alibi} yields in the weak sense,
\[
\partial_t\upmu_N(t,\upsigma,\uptheta)
=
-\partial_\uptheta\left(\upmu_N(t,\upsigma,\uptheta)\int_0^\upsigma\int_\T \mathsf{K}_N(\upsigma,\upsigma')\mathsf{w}_{\upbeta}'(\uptheta-\uptheta')\,\upmu_N(t,\mathrm d\upsigma',\mathrm d\uptheta')\right).
\]
By weak compactness, up to a subsequence, we have $\upmu_N(t)\overset*\rightharpoonup f(t)$ for all $t\ge0$. Passing to the limit in the weak formulation of the above equation, using the graphon convergence $\mathsf{K}_N\to \mathsf{k}_{\uplambda}$, cf.~\eqref{eq:estim-KNk}, we deduce that the limit point $f$ is necessarily a weak solution of~\eqref{eq:mfl-lambda}. In addition, noting that $\int_\T \upmu_N(t,\upsigma,\mathrm d\uptheta)=\frac1N\sum_{i=1}^N\updelta_{i/N}(\upsigma)\overset*\rightharpoonup1$, we deduce $f\in L^\infty_\loc(\mathbb{R}_{\geqslant0}\times[0,1];\Pc(\T))$.

It remains to show that a weak solution of~\eqref{eq:mfl-lambda} in $L^\infty_\loc(\mathbb{R}_{\geqslant0}\times[0,1];\Pc(\T))$ is necessarily unique. Let~$f,f'$ be two such solutions with identical initial data. For almost all $\upsigma$, as $f(\cdot,\upsigma,\cdot)$ and $f'(\cdot,\upsigma,\cdot)$ satisfy transport equations with velocity fields $V_{\upsigma}$ and $V_{\upsigma}'$, which are given by
\[V_{\upsigma}(t,\uptheta)=-\int_0^\upsigma \mathsf{k}_{\uplambda}(\upsigma,\uptau)\mathsf{w}_{\upbeta}'\ast_\uptheta f(t,\uptau,\uptheta)\,\mathrm d\uptau\]
and correspondingly for $V_{\upsigma}'$.
A standard calculation yields
\[\partial_t^+W_1(f(t,\upsigma),f'(t,\upsigma))\le\|\nabla V_{\upsigma}'(t)\|_{L^\infty(\T)}W_1(f(t,\upsigma),f'(t,\upsigma))+\int_\T|V_{\upsigma}(t)-V_{\upsigma}'(t)|f(t,\upsigma),\]
where $W_1$ denotes the usual $1$-Wasserstein distance on $\Pc(\T)$.
Inserting the definition of $V_{\upsigma},V_{\upsigma}'$, we deduce
\begin{equation*}
{\partial_t^+W_1(f(t,\upsigma),f'(t,\upsigma))}
\le\|\mathsf{w}_{\upbeta}''\|_{L^\infty(\T)}\left(W_1(f(t,\upsigma),f'(t,\upsigma))+\int_0^\upsigma \mathsf{k}_{\uplambda}(\upsigma,\uptau)W_1(f(t,\uptau),f'(t,\uptau))\,\mathrm d\uptau\right),
\end{equation*}
and thus,
\[\partial_t^+\left(\underset{\upsigma\in[0,1]}\supess\, W_1(f(t,\upsigma),f'(t,\upsigma))\right)\,\le\,2\|\mathsf{w}_{\upbeta}''\|_{L^\infty(\T)}\,\underset{\upsigma\in[0,1]}\supess\, W_1(f(t,\upsigma),f'(t,\upsigma)).\]
By Gr\"onwall's inequality, this entails $f=f'$, which concludes the proof of~(i).

\medskip
\step2~Approximate equation for marginal distribution: for $f_N$ defined in~\eqref{eq:marg-fN}, we have for all $t\ge0$, $\upsigma\in(0,1]$, and $n\in\Z$,
\begin{equation}\label{eq:approx-hatfN}
\left|\partial_t\hat f_N(t,\upsigma,n)
-\sum_{\upxi\in\Z}n\upxi\widehat{\mathsf{w}}_{\upbeta}(\upxi) \hat f_N(t,\upsigma,n-\upxi)\int_0^\upsigma \mathsf{k}_{\uplambda}(\upsigma,\uptau)\hat f_N(t,\uptau,\upxi)\,\mathrm d\uptau\right|
\le \frac{C}{N\upsigma+1} e^{C t}\langle n\rangle^2.
\end{equation}
With the short-hand notation~\eqref{eq:shortdef-omega}, the particle dynamics~\eqref{eq:gpt-dyn-alibi} yields
\[
\frac{\mathrm d}{\mathrm dt}\E\left[e^{-in\uptheta_j}\right]
=
-in\sum_{k=1}^{j-1}\upomega_{j,k}\E\left[e^{-in\uptheta_j}\mathsf{w}_{\upbeta}'(\uptheta_j-\uptheta_k)\right],
\]
and thus, using Fourier decomposition for the interaction $\mathsf{w}_{\upbeta}'$,
\[
\frac{\mathrm d}{\mathrm dt}\E\left[e^{-in\uptheta_j}\right]
=
\sum_{\upxi\in\Z}n\upxi\widehat{\mathsf{w}}_{\upbeta}(\upxi)\sum_{k=1}^{j-1}\upomega_{j,k}\E\left[e^{-i(n-\upxi)\uptheta_j}e^{-i\upxi\uptheta_k}\right].
\]
Using the smallness of correlations obtained in Lemma~\ref{lem:correl}, and recalling $\upomega_{j,k}\lesssim j^{-1}$, we deduce
\begin{multline*}
\left|\frac{\mathrm d}{\mathrm dt}\E\left[e^{-in\uptheta_j}\right]
-\sum_{\upxi\in\Z}n\upxi\widehat{\mathsf{w}}_{\upbeta}(\upxi)\sum_{k=1}^{j-1}\upomega_{j,k}\E\left[e^{-i(n-\upxi)\uptheta_j}\right]\,\E\left[e^{-i\upxi\uptheta_k}\right]\right|\\
\le C j^{-1}\sum_{k=1}^{j-1}\upomega_{j,k}e^{\sqrt{C t\log(j/k)}}e^{C t}\langle n\rangle^2
\le C j^{-1}e^{C t}\langle n\rangle^2.
\end{multline*}
For $\upsigma\in(0,1]$ and $j=\lceil N\upsigma\rceil$, in terms of Fourier modes $\hat f_N(t,\upsigma,n)=\E\big[e^{-in\uptheta_{\lceil N\upsigma\rceil}(t)}\big]$ and of the kernel~$\mathsf{K}_N$ defined in~\eqref{eq:def-KN}, this reads
\[
\left|\partial_t\hat f_N(t,\upsigma,n)
-\sum_{\upxi\in\Z}n\upxi\widehat{\mathsf{w}}_{\upbeta}(\upxi)\,\hat f_N(t,\upsigma,n-\upxi)\int_0^1\mathsf{K}_N(\upsigma,\upsigma')\hat f_N(t,\upsigma',\upxi)\mathrm d\upsigma'\right|
\le \frac{C}{N\upsigma+1}e^{C t}\langle n\rangle^2.
\]
Using the graphon convergence~\eqref{eq:estim-KNk} together with the trivial bound $|\hat f_N|\le1$,
this concludes the proof of \eqref{eq:approx-hatfN}.

\medskip
\step3~Stability for limit equation.\\
Based on the approximate equation~\eqref{eq:approx-hatfN}, we aim to deduce an error estimate between the marginal distribution $f_N$ and the weak solution $f$ of the limit equation~\eqref{eq:mfl-lambda}. 
First note that the latter reads as follows in Fourier space,
\[
\partial_t\hat f(t,\upsigma,n)-\sum_{\upxi\in\Z}n\upxi \widehat{\mathsf{w}}_{\upbeta}(\upxi)\hat f(t,\upsigma,n-\upxi)\int_0^\upsigma \mathsf{k}_{\uplambda}(\upsigma,\uptau)\hat f(t,\uptau,\upxi)\,\mathrm d\uptau=0.
\]
Hence, comparing this with~\eqref{eq:approx-hatfN}, we find for the error $g_N:=f_N-f$,
\begin{multline}\label{eq:err-gN}
\partial_t\hat g_N(t,\upsigma,n)
-\sum_{\upxi\in\Z}n\upxi\widehat{\mathsf{w}}_{\upbeta}(\upxi) \hat f(t,\upsigma,n-\upxi)\int_0^\upsigma \mathsf{k}_{\uplambda}(\upsigma,\uptau)\hat g_N(t,\uptau,\upxi)\,\mathrm d\uptau\\
-\sum_{\upxi\in\Z}n\upxi\widehat{\mathsf{w}}_{\upbeta}(\upxi) \hat g_N(t,\upsigma,n-\upxi)\int_0^\upsigma \mathsf{k}_{\uplambda}(\upsigma,\uptau)\hat f_N(t,\uptau,\upxi)\,\mathrm d\uptau
=\hat r_N(t,\upsigma,n),
\end{multline}
where the remainder term $r_N$ satisfies pointwise, for all $t\ge0$, $\upsigma\in(0,1]$, and $n\in \Z$,
\begin{equation}\label{eq:bound-rN}
|\hat r_N(t,\upsigma,n)|\,\le\,\frac{C}{N\upsigma+1} e^{C t}\langle n\rangle^2.
\end{equation}
Due to the loss of a derivative in $\uptheta$ (that is, of a factor $n$) in~\eqref{eq:err-gN}, we cannot prove an $L^\infty$ stability estimate for Fourier modes $\hat g_N(t,\upsigma,n)$, and we shall rather work in a weighted $L^2$ setting.
Testing~\eqref{eq:err-gN} with $\langle n\rangle^{-\upalpha}\hat g_N(t,\upsigma,n)$, for some exponent $\upalpha\ge0$, we find
\begin{multline*}
\partial_t\sum_{n\in\Z}\langle n\rangle^{-\upalpha}|\hat g_N(t,\upsigma,n)|^2
\le2\left|\sum_{n\in\Z}\langle n\rangle^{-\upalpha}\hat g_N(t,\upsigma,-n)\hat r_N(t,\upsigma,n)\right|\\
-2\left|\sum_{n,\upxi\in\Z}\langle n\rangle^{-\upalpha}n\upxi\widehat{\mathsf{w}}_{\upbeta}(\upxi) \hat g_N(t,\upsigma,-n) \hat f(t,\upsigma,n-\upxi)\int_0^\upsigma \mathsf{k}_{\uplambda}(\upsigma,\uptau)\hat g_N(t,\uptau,\upxi)\,\mathrm d\uptau\right|\\
-2\left|\sum_{n,\upxi\in\Z}\langle n\rangle^{-\upalpha}n\upxi\widehat{\mathsf{w}}_{\upbeta}(\upxi) \hat g_N(t,\upsigma,-n)\hat g_N(t,\upsigma,n-\upxi)\int_0^\upsigma \mathsf{k}_{\uplambda}(\upsigma,\uptau)\hat f_N(t,\uptau,\upxi)\,\mathrm d\uptau\right|.
\end{multline*}
The sum over $n$ in the last term can be reorganized as follows, using the symmetry $n\leftrightarrow\upxi-n$,
\begin{eqnarray}
\lefteqn{\sum_{n\in\Z}\langle n\rangle^{-\upalpha}n \,\hat g_N(t,\upsigma,-n)\,\hat g_N(t,\upsigma,n-\upxi)}\nonumber\\
&=&\frac12\sum_{n\in\Z}\left(\langle n-\upxi\rangle^{-\upalpha}\upxi+\left(\langle n\rangle^{-\upalpha} - \langle n-\upxi\rangle^{-\upalpha}\right)n\right)\,\hat g_N(t,\upsigma,-n)\,\hat g_N(t,\upsigma,n-\upxi)\nonumber\\
&\lesssim&\langle\upxi\rangle^{\upalpha+2}\sum_{n\in\Z}\langle n\rangle^{-\upalpha}|\hat g_N(t,\upsigma,n)||\hat g_N(t,\upsigma,n-\upxi)|.\label{eq:estim-nxi}
\end{eqnarray}
Further using $\mathsf{k}_{\uplambda}(\upsigma,\uptau)\lesssim \upsigma^{-1}$, the above then becomes
\begin{multline*}
\partial_t\left(\sum_{n\in\Z}\langle n\rangle^{-\upalpha}|\hat g_N(t,\upsigma,n)|^2\right)^{\frac12}
\lesssim\left(\sum_{n\in\Z}\langle n\rangle^{-\upalpha}|\hat r_N(t,\upsigma,n)|^2\right)^{\frac12}\\
+\left(\sum_{n\in\Z}\langle n\rangle^{2-\upalpha}|\hat f(t,\upsigma,n)|^2\right)^{\frac12}\frac1\upsigma\int_0^\upsigma \left(\sum_{n\in\Z}\langle n\rangle^{-\upalpha}|\hat g_N(t,\uptau,n)|^2\right)^{\frac12}\,\mathrm d\uptau\\
+\left(\sum_{n\in\Z}\langle n\rangle^{-\upalpha}|\hat g_N(t,\upsigma,n)|^2\right)^{\frac12}\frac1\upsigma\int_0^\upsigma \left(\sum_{n\in\Z}\langle n\rangle^{-\upalpha}|\hat f_N(t,\uptau,n)|^2\right)^{\frac12}\,\mathrm d\uptau.
\end{multline*}
Using the pointwise bound~\eqref{eq:bound-rN} on the remainder $r_N$, together with the trivial bounds $|\hat f|,|\hat f_N|\le1$, and choosing $\upalpha>5$, we are led to
\begin{multline}\label{eq:estim-gN-pre}
\partial_t\left(\sum_{n\in\Z}\langle n\rangle^{-\upalpha}|\hat g_N(t,\upsigma,n)|^2\right)^{\frac12}\\
\lesssim\frac1{N\upsigma+1}e^{C t}
+\left(\sum_{n\in\Z}\langle n\rangle^{-\upalpha}|\hat g_N(t,\upsigma,n)|^2\right)^{\frac12}
+\frac1\upsigma\int_0^\upsigma \left(\sum_{n\in\Z}\langle n\rangle^{-\upalpha}|\hat g_N(t,\uptau,n)|^2\right)^{\frac12}\,\mathrm d\uptau.
\end{multline}
Recall the Hardy inequality: for every $h\in L^p(0,1)$ and $1<p<\infty$,
\[
\left(\int_0^1\left|\frac{1}{\upsigma}\int_0^\upsigma h(\uptau)\,\mathrm d\uptau\right|^p\,\mathrm d\upsigma\right)^{\frac1p}
\le \frac{p}{p-1}\left(\int_0^1\left|h(\upsigma)\right|^p\,\mathrm d\upsigma\right)^{\frac1p}.
\]
Using this, we get
\begin{equation*}
\partial_t\left\|\left(\sum_{n\in\Z}\langle n\rangle^{-\upalpha}|\hat g_N(t,\cdot,n)|^2\right)^{\frac12}\right\|_{L^p(0,1)}
\lesssim N^{-\frac1p}e^{C t}
+\left\|\left(\sum_{n\in\Z}\langle n\rangle^{-\upalpha}|\hat g_N(t,\cdot,n)|^2\right)^{\frac12}\right\|_{L^p(0,1)},
\end{equation*}
and thus, by Gr\"onwall's inequality, recalling the initial assumption~\eqref{eq:init-conv} and choosing $\upalpha>2\upgamma+1$,
\begin{equation}\label{eq:estim-gN-res}
\left\|\left(\sum_{n\in\Z}\langle n\rangle^{-\upalpha}|\hat g_N(t,\cdot,n)|^2\right)^{\frac12}\right\|_{L^p(0,1)}
\lesssim N^{-\updelta\wedge\frac1p}\,e^{C t}.
\end{equation}

\medskip
\step4~Conclusion: proof of error estimates.\\
We start by noticing the link between the empirical measure $\upmu_N$ and the marginal distribution $f_N$.
On the one hand, we note that $f_N$ approximates the expectation of the empirical measure: by definition of $\upmu_N$ and $f_N$, we find for any $\upvarphi\in C^\infty((0,1)\times\T)$,
\[\E\left[\int_{(0,1]\times\T}\upvarphi\upmu_N(t)\right]-\int_{(0,1]\times\T}\upvarphi f_N(t)=\int_0^1\E\left[\upvarphi\left(\frac{\lceil N\upsigma\rceil}{N},\uptheta_{\lceil N\upsigma\rceil}(t)\right)-\upvarphi\left(\upsigma,\uptheta_{\lceil N\upsigma\rceil}(t)\right)\right]\,\mathrm d\upsigma,\]
and thus,
\[\left|\E\left[\int_{(0,1]\times\T}\upvarphi\upmu_N(t)\right]-\int_{(0,1]\times\T}\upvarphi f_N(t)\right|\lesssim N^{-1}\|\partial_\upsigma\upvarphi\|_{L^\infty((0,1)\times\T)}.\]
On the other hand, the variance of the empirical measure can be expanded as
\begin{multline*}
\Var\left[\int_{(0,1]\times\T}\upvarphi\upmu_N(t)\right]\\
=2N^{-2}\sum_{i>j}^N\Cov\left(\upvarphi\left(\frac iN,\uptheta_i(t)\right),\upvarphi\left(\frac jN,\uptheta_j(t)\right)\right)+N^{-2}\sum_{i=1}^N\Var\left[\upvarphi\left(\frac iN,\uptheta_i(t)\right)\right],
\end{multline*}
and thus, by the smallness of correlations obtained in Lemma~\ref{lem:correl},
\begin{eqnarray*}
\Var\left[\int_{(0,1]\times\T}\upvarphi\upmu_N(t)\right]
&\le&C N^{-2}\sum_{i> j}^Ni^{-1}e^{\sqrt{C t\log(i/j)}}e^{C t}\|\upvarphi\|_{L^\infty((0,1];W^{1,\infty}(\T))}^2+N^{-1}\|\upvarphi\|_{L^\infty((0,1]\times\T)}^2\\
&\le&C N^{-1}e^{C t}\|\upvarphi\|_{L^\infty((0,1];W^{1,\infty}(\T))}^2.
\end{eqnarray*}
Combining these two estimates, we deduce
\begin{equation*}
\E\left[\left|\int_{(0,1]\times\T}\upvarphi(\upmu_N(t)-f(t))\right|^2\right]\\
\lesssim
\left|\int_{(0,1)\times\T}\upvarphi(f_N(t)-f(t))\right|^2
+N^{-1}e^{C t}\|\upvarphi\|_{W^{1,\infty}((0,1)\times\T)}^2.
\end{equation*}
Finally, combining this with the bound~\eqref{eq:estim-gN-res} on $g_N=f_N-f$, choosing e.g.\@ $p=2$, the conclusion~\eqref{eq:conv-rate-muNf} follows.
\end{proof}

\begin{rem}[Convergence of marginal distribution]\label{rem:marginal-convergence}
For future reference, we show that the error bound~\eqref{eq:estim-gN-res} for $g_N=f_N-f$ can be upgraded to a pointwise estimate.
Indeed, using~\eqref{eq:estim-gN-res} to bound the last right-hand side term in~\eqref{eq:estim-gN-pre}, for $\upalpha>(2\upgamma)\vee4+1$ and $1<p<\infty$, we find
\begin{eqnarray*}
\lefteqn{\partial_t\left(\sum_{n\in\Z}\langle n\rangle^{-\upalpha}|\hat g_N(t,\upsigma,n)|^2\right)^{\frac12}}\\
&\lesssim&\frac1{N\upsigma+1}e^{C t}
+\left(\sum_{n\in\Z}\langle n\rangle^{-\upalpha}|\hat g_N(t,\upsigma,n)|^2\right)^{\frac12}
+\upsigma^{-\frac1p}\left\|\left(\sum_{n\in\Z}\langle n\rangle^{-\upalpha}|\hat g_N(t,\cdot,n)|^2\right)^{\frac12}\right\|_{L^p((0,1))}\\
&\lesssim&
N^{-\updelta\wedge\frac1p}\upsigma^{-\frac1p}e^{C t}
+\left(\sum_{n\in\Z}\langle n\rangle^{-\upalpha}|\hat g_N(t,\upsigma,n)|^2\right)^{\frac12},
\end{eqnarray*}
and thus, by Gr\"onwall's inequality with the initial assumption~\eqref{eq:init-conv},
\begin{eqnarray*}
\left(\sum_{n\in\Z}\langle n\rangle^{-\upalpha}|\hat g_N(t,\upsigma,n)|^2\right)^{\frac12}
\lesssim N^{-\updelta\wedge\frac1p}\upsigma^{-\frac1p}e^{C t}.
\end{eqnarray*}
Hence, for any $\upzeta\le \updelta$ with $\upzeta<1$, we have for all $t\ge0$, $\upsigma\in(0,1]$, and $n\in\Z$,
\begin{equation}\label{eq:estim-gN-res-re}
|\hat f_N(t,\upsigma,n)-\hat f(t,\upsigma,n)|\lesssim (N\upsigma)^{-\upzeta}e^{C t}\langle n\rangle^C.
\end{equation}
\end{rem}

\section{Characterization of correlations}\label{sec:corr-proof}

This section proves Theorem~\ref{th:lost}. The argument proceeds by using the cumulant estimates of Section~\ref{sec:cumulants} to derive approximate closed equations for autocorrelations and cross-correlations. The conclusion then follows from a suitable stability analysis for the latter, similarly as in Step~3 of the proof of Therem~\ref{th:mf-gpt}.

\begin{proof}[Proof of Theorem~\ref{th:lost}]
We split the argument into three steps.
\medskip

\step1~Approximate equations for time correlations:
\begin{enumerate}[---]
\item autocorrelations: for all $t\ge0$, $\upsigma\in(0,1]$, and $n\in\Z$, 
\begin{multline}\label{eq:appr-autocor}
\qquad\left|\partial_t\hat A_\upvarphi^N(t,\upsigma,n)
-\sum_{\upxi\in\Z} n\upxi\widehat{\mathsf{w}}_{\upbeta}(\upxi)\,\hat A_\upvarphi^N(t,\upsigma,n-\upxi)\int_0^\upsigma \mathsf{k}_{\uplambda}(\upsigma,\uptau)\hat f_N(t,\uptau,\upxi)\,\mathrm d\uptau\right|\\
\lesssim
\frac1{N\upsigma^2} e^{C t}\langle n\rangle^3\|\upvarphi\|_{W^{2,\infty}(\T)};
\end{multline}
\item cross-correlations: for all $t\ge0$, $\upsigma,\upsigma_0\in(0,1]$, and $n\in\Z$,
\begin{multline}\label{eq:appr-crosscor}
\qquad\left|\partial_t(N\hat C_\upvarphi^N)(t,\upsigma,n;\upsigma_0)-\sum_{\upxi\in\Z} n\upxi\widehat{\mathsf{w}}_{\upbeta}(\upxi)\,\hat f_N(t,\upsigma,n-\upxi)\int_{\upsigma_0}^\upsigma \mathsf{k}_{\uplambda}(\upsigma,\uptau) (N\hat C_\upvarphi^N)(t,\uptau,\upxi;\upsigma_0)\,\mathrm d\uptau\right.\\
\left.-\sum_{\upxi\in\Z} n\upxi\widehat{\mathsf{w}}_{\upbeta}(\upxi)\,(N\hat C_\upvarphi^N)(t,\upsigma,n-\upxi;\upsigma_0)\int_0^\upsigma \mathsf{k}_{\uplambda}(\upsigma,\uptau)\hat f_N(t,\uptau,\upxi)\,\mathrm d\uptau\right.\\
\left.-\sum_{\upxi\in\Z} n\upxi\widehat{\mathsf{w}}_{\upbeta}(\upxi)\hat f_N(t,\upsigma,n-\upxi)\,\mathsf{k}_{\uplambda}(\upsigma,\upsigma_0)\hat A_\upvarphi^N(t,\upsigma_0,\upxi)\right|\\
\lesssim
\frac1{N\upsigma^2} e^{\sqrt{C t\log(\upsigma/\upsigma_0)}}e^{C t}\langle n\rangle^3\|\upvarphi\|_{W^{2,\infty}(\T)}.
\end{multline}
\end{enumerate}
Let $\ell\le j$ and $n\in\Z$ be fixed.
By the particle dynamics~\eqref{eq:gpt-dyn-alibi}, using Fourier decomposition for the interaction $\mathsf{w}_{\upbeta}'$, we can compute
\begin{eqnarray}
\frac{\mathrm d}{\mathrm dt}\Cov\left(e^{in\uptheta_j},\upvarphi\left(\uptheta_\ell^0\right)\right)
&=&in\sum_{k=1}^{j-1}\upomega_{j,k}\Cov\left(e^{in\uptheta_j}\mathsf{w}_{\upbeta}'\left(\uptheta_j-\uptheta_k\right),\,\upvarphi\left(\uptheta_\ell^0\right)\right)\nonumber\\
&=&-\sum_{\upxi\in\Z} n\upxi\widehat{\mathsf{w}}_{\upbeta}(\upxi)\sum_{k=1}^{j-1}\upomega_{j,k}\Cov\left(e^{i(n+\upxi)\uptheta_j}e^{-i\upxi\uptheta_k},\,\upvarphi\left(\uptheta_\ell^0\right)\right).\label{eq:est-dtcov-0}
\end{eqnarray}
By the law of total cumulance, we may expand, for $k\ne \ell$,
\begin{eqnarray*}
\lefteqn{\Cov\left(e^{i(n+\upxi)\uptheta_j}e^{-i\upxi\uptheta_k},\,\upvarphi\left(\uptheta_\ell^0\right)\right)}\\
&=&\E\left[e^{i(n+\upxi)\uptheta_j}e^{-i\upxi\uptheta_k}\upvarphi\left(\uptheta_\ell^0\right)\right]-\E\left[e^{i(n+\upxi)\uptheta_j}e^{-i\upxi\uptheta_k}\right]\,\E\left[\upvarphi\left(\uptheta_\ell^0\right)\right]\\
&=&\E\left[e^{i(n+\upxi)\uptheta_j}\right]\,\Cov\left(e^{-i\upxi\uptheta_k},\upvarphi\left(\uptheta_\ell^0\right)\right)
+\E\left[e^{-i\upxi\uptheta_k}\right]\,\Cov\left(e^{i(n+\upxi)\uptheta_j},\upvarphi\left(\uptheta_\ell^0\right)\right)\\
&+&\upkappa_{1,1,1}\left(e^{i(n+\upxi)\uptheta_j},e^{-i\upxi\uptheta_k},\upvarphi\left(\uptheta_\ell^0\right)\right),
\end{eqnarray*}
and thus, appealing to the cumulant estimate of Lemma~\ref{lem:correl},
\begin{multline*}
\left|\Cov\left(e^{i(n+\upxi)\uptheta_j}e^{-i\upxi\uptheta_k},\,\upvarphi\left(\uptheta_\ell^0\right)\right)\right.\\
\left.-\E\left[e^{i(n+\upxi)\uptheta_j}\right]\,\Cov\left(e^{-i\upxi\uptheta_k},\upvarphi\left(\uptheta_\ell^0\right)\right)
-\E\left[e^{-i\upxi\uptheta_k}\right]\,\Cov\left(e^{i(n+\upxi)\uptheta_j},\upvarphi\left(\uptheta_\ell^0\right)\right)\right|\\
\lesssim j^{-1}(k\vee \ell)^{-1}e^{\sqrt{C t\log(j/(k\wedge \ell))}}e^{C t}\langle\upxi\rangle^3\langle n\rangle^2\|\upvarphi\|_{W^{2,\infty}(\T)}.
\end{multline*}
Similarly, for $k=\ell$, we find
\begin{multline*}
\left|\Cov\left(e^{i(n+\upxi)\uptheta_j}e^{-i\upxi\uptheta_\ell},\,\upvarphi\left(\uptheta_\ell^0\right)\right)
-\E\left[e^{i(n+\upxi)\uptheta_j}\right]\,\Cov\left(e^{-i\upxi\uptheta_\ell},\upvarphi\left(\uptheta_\ell^0\right)\right)\right|\\
\lesssim j^{-1}e^{\sqrt{C t\log(j/\ell)}}e^{C t}\langle\upxi\rangle^2\langle n\rangle\|\upvarphi\|_{W^{1,\infty}(\T)}.
\end{multline*}
Inserting these estimates into~\eqref{eq:est-dtcov-0}, and recalling that $\Cov\left(e^{-i\upxi\uptheta_k},\upvarphi\left(\uptheta_\ell^0\right)\right)=0$ for $k<\ell$ by the triangular structure of the dynamics, we deduce for all $t\ge0$,
\begin{multline*}
\left|\frac{\mathrm d}{\mathrm dt}\Cov\left(e^{in\uptheta_j},\upvarphi\left(\uptheta_\ell^0\right)\right)
-\sum_{\upxi\in\Z} n\upxi\widehat{\mathsf{w}}_{\upbeta}(\upxi)\,\E\left[e^{i(n+\upxi)\uptheta_j}\right]\sum_{\ell<k<j}\upomega_{j,k}\Cov\left(e^{-i\upxi\uptheta_k},\upvarphi\left(\uptheta_\ell^0\right)\right)\right.\\
\left.-\sum_{\upxi\in\Z} n\upxi\widehat{\mathsf{w}}_{\upbeta}(\upxi)\,\Cov\left(e^{i(n+\upxi)\uptheta_j},\upvarphi\left(\uptheta_\ell^0\right)\right)\sum_{1\le k<j\atop k\ne \ell}\upomega_{j,k}\E\left[e^{-i\upxi\uptheta_k}\right]\right.\\
\left.-\mathds1_{\ell<j}\sum_{\upxi\in\Z} n\upxi\widehat{\mathsf{w}}_{\upbeta}(\upxi)\upomega_{j,\ell}\E\left[e^{i(n+\upxi)\uptheta_j}\right]\,\Cov\left(e^{-i\upxi\uptheta_\ell},\upvarphi\left(\uptheta_\ell^0\right)\right)\right|
\lesssim
j^{-2}e^{\sqrt{C t\log(j/\ell)}}e^{C t}\langle n\rangle^3\|\upvarphi\|_{W^{2,\infty}(\T)}.
\end{multline*}
Separately considering the cases $\ell<j$ and $\ell=j$, recalling the covariance estimate of Lemma~\ref{lem:correl}, using the graphon convergence~\eqref{eq:estim-KNk},
and recalling the definition of correlation functions $A_\upvarphi^N,C_\upvarphi^N$, cf.~\eqref{eq:def-Covn}--\eqref{eq:def-Covn-2}, and of the marginal distribution $f_N$, cf.~\eqref{eq:marg-fN},
this implies the claimed approximate equations~\eqref{eq:appr-autocor}--\eqref{eq:appr-crosscor}.

\medskip
\step2~ Error estimate for autocorrelations.\\
Using the mean-field error estimate~\eqref{eq:estim-gN-res-re} for the marginal distribution $f_N$, the approximate equation~\eqref{eq:appr-autocor} becomes, for $\upzeta\le\updelta$ with $\upzeta<1$,
\begin{equation*}
\left|\partial_t\hat A_\upvarphi^N(t,\upsigma,n)
-\sum_{\upxi\in\Z} n\upxi\widehat{\mathsf{w}}_{\upbeta}(\upxi)\,\hat A_\upvarphi^N(t,\upsigma,n-\upxi)\int_0^\upsigma \mathsf{k}_{\uplambda}(\upsigma,\uptau)\hat f(t,\uptau,\upxi)\,\mathrm d\uptau\right|
\lesssim N^{-\upzeta}\upsigma^{-2} e^{C t}\langle n\rangle^C\|\upvarphi\|_{W^{2,\infty}(\T)}.
\end{equation*}
Comparing with the solution $A_\upvarphi$ of the limit equation~\eqref{eq:Aphi-lambda}, and using symmetry in $n$ as in~\eqref{eq:estim-nxi}, we obtain for $\upalpha\gg1$,
\begin{equation*}
\partial_t\left(\sum_n\langle n\rangle^{-\upalpha}|(\hat A_\upvarphi^N-\hat A_\upvarphi)(t,\upsigma,n)|^2\right)^{\frac12}
\lesssim \left(\sum_n\langle n\rangle^{-\upalpha}|(\hat A_\upvarphi^N-\hat A_\upvarphi)(t,\upsigma,n)|^2\right)^{\frac12}
+N^{-\upzeta}\upsigma^{-2} e^{C t}\|\upvarphi\|_{W^{2,\infty}(\T)},
\end{equation*}
and thus, by Gr\"onwall's inequality,
\begin{multline}\label{eq:estim-AN-A-pre}
\left(\sum_n\langle n\rangle^{-\upalpha}|(\hat A_\upvarphi^N-\hat A_\upvarphi)(t,\upsigma,n)|^2\right)^{\frac12}
\lesssim e^{C t}\left(\sum_n\langle n\rangle^{-\upalpha}|(\hat A_\upvarphi^N-\hat A_\upvarphi)(0,\upsigma,n)|^2\right)^{\frac12}\\[-1mm]
+N^{-\upzeta}\upsigma^{-2} e^{C t}\|\upvarphi\|_{W^{2,\infty}(\T)}.
\end{multline}
By definition~\eqref{eq:def-Covn}, we can write the initial autocorrelation in terms of the marginal distribution,
\[\hat A_\upvarphi^N(0,\upsigma,n)= \int_\T e^{-in\uptheta}f_N(0,\upsigma,\uptheta)\left(\upvarphi(\uptheta)-\int_\T\upvarphi f_N(0,\upsigma,\cdot)\right) \,\mathrm d\uptheta.\]
Comparing this with the initial condition for $A_\upvarphi$ in~\eqref{eq:Aphi-lambda} and computing the Fourier coefficient, we obtain for $\upsigma\in(0,1]$,
\begin{multline*}
(\hat A_\upvarphi^N-\hat A_\upvarphi)(0,\upsigma,n)
=\sum_{\upxi\in\Z}\left(\hat\upvarphi(\upxi)\,(\hat f_N-\hat f)(0,\upsigma,n-\upxi)
-\hat\upvarphi(\upxi)\hat f_N(0,\upsigma,-\upxi)(\hat f_N-\hat f)(0,\upsigma,n)\right.\\[-3mm]
\left.-\hat\upvarphi(\upxi)\hat f(0,\upsigma,n)(\hat f_N-\hat f)(0,\upsigma,-\upxi)\right),
\end{multline*}
and thus, by the initial convergence assumption~\eqref{eq:init-conv},
\begin{equation*}
|(\hat A_\upvarphi^N-\hat A_\upvarphi)(0,\upsigma,n)|
\lesssim N^{-\updelta}\langle n\rangle^\upgamma\sum_{\upxi\in\Z}\langle \upxi\rangle^\upgamma|\hat\upvarphi(\upxi)|.
\end{equation*}
Inserting this into~\eqref{eq:estim-AN-A-pre}, the conclusion~\eqref{eq:estim-ANA} follows.

\medskip
\step3~Error estimate for cross-correlations.\\
Using the mean-field error estimate~\eqref{eq:estim-gN-res-re} for the marginal distribution $f_N$, as well as~\eqref{eq:estim-ANA} for autocorrelations, and recalling the correlation estimates of Lemma~\ref{lem:correl}, we can replace $\hat f_N,\hat A_\upvarphi^N$ by $\hat f,\hat A_\upvarphi$ in the approximate equation~\eqref{eq:appr-crosscor}: for $\upzeta\le\updelta$ with $\upzeta<1$,
\begin{multline*}
\qquad\left|\partial_t(N\hat C_\upvarphi^N)(t,\upsigma,n;\upsigma_0)
-\sum_{\upxi\in\Z} n\upxi\widehat{\mathsf{w}}_{\upbeta}(\upxi)\,\hat f(t,\upsigma,n-\upxi)\int_{\upsigma_0}^\upsigma \mathsf{k}_{\uplambda}(\upsigma,\uptau) (N\hat C_\upvarphi^N)(t,\uptau,\upxi;\upsigma_0)\,\mathrm d\uptau\right.\\
\left.-\sum_{\upxi\in\Z} n\upxi\widehat{\mathsf{w}}_{\upbeta}(\upxi)\,(N\hat C_\upvarphi^N)(t,\upsigma,n-\upxi;\upsigma_0)\int_0^\upsigma \mathsf{k}_{\uplambda}(\upsigma,\uptau)\hat f(t,\uptau,\upxi)\,\mathrm d\uptau\right.\\
\left.-\sum_{\upxi\in\Z} n\upxi\widehat{\mathsf{w}}_{\upbeta}(\upxi)\hat f(t,\upsigma,n-\upxi)\,\mathsf{k}_{\uplambda}(\upsigma,\upsigma_0)\hat A_\upvarphi(t,\upsigma_0,\upxi)\right|
\lesssim_{\upvarphi}
\frac1{N^\upzeta\upsigma^2} e^{\sqrt{C t\log(\upsigma/\upsigma_0)}}e^{C t}\langle n\rangle^C.
\end{multline*}
Comparing with the solution $C_\upvarphi$ of the limit equation~\eqref{eq:Cphi-lambda}, and using symmetry in $n$ as in~\eqref{eq:estim-nxi}, we obtain for $\upalpha\gg1$,
\begin{multline*}
\partial_t\left(\sum_{n\in\Z}\langle n\rangle^{-\upalpha}|(N\hat C_\upvarphi^N-\hat C_\upvarphi)(t,\upsigma,n;\upsigma_0)|^2\right)^{\frac12}\\[-2mm]
\lesssim
\int_{\upsigma_0}^\upsigma \mathsf{k}_{\uplambda}(\upsigma,\uptau)\left(\sum_{n\in\Z}\langle n\rangle^{-\upalpha}|(N\hat C_\upvarphi^N-\hat C_\upvarphi)(t,\uptau,n;\upsigma_0)|^2\right)^{\frac12}\,\mathrm d\uptau\\
+\left(\sum_{n\in\Z}\langle n\rangle^{-\upalpha}|(N\hat C_\upvarphi^N-\hat C_\upvarphi)(t,\upsigma,n;\upsigma_0)|^2\right)^{\frac12}
+\frac{C(\upvarphi)}{N^\upzeta\upsigma^2} e^{\sqrt{C t\log(\upsigma/\upsigma_0)}}e^{C t},
\end{multline*}
and thus, by Gr\"onwall's inequality, with $\hat C_\upvarphi^N|_{t=0}=\hat C_\upvarphi|_{t=0}=0$,
\begin{equation*}
\left(\sum_{n\in\Z}\langle n\rangle^{-\upalpha}|(N\hat C_\upvarphi^N-\hat C_\upvarphi)(t,\upsigma,n;\upsigma_0)|^2\right)^{\frac12}
\lesssim_{\upvarphi}\frac1{N^\upzeta\upsigma_0^2} e^{C t},
\end{equation*}
which concludes the proof of~\eqref{eq:estim-CNC}.
\end{proof}

\section{Lost in the middle}\label{sec:litm}

This section is devoted to the proof of
Theorem~\ref{thm:U-shape}.
We work in the iid uniform case $f\equiv1$, $A_\upvarphi\equiv\upvarphi-\int_\T\upvarphi$, with fixed $\uplambda,\upbeta>0$. By Theorem~\ref{th:lost},
we recall that the soft accuracy admits the expansion~\eqref{eq:soft-accuracy-expansion}, so that it remains to study the profile of the leading correction $\mathscr{S}_t(\upsigma_0)$ defined in~\eqref{eq:Sb-def}.
The argument relies on the explicit solvability of the Volterra--Hardy equation~\eqref{eq:Hardy-hom}, as stated in Lemma~\ref{prop:Bessel}.
For this, we reduce the equation to a Goursat problem and solve it explicitly in terms of modified Bessel functions.

\begin{proof}[Proof of Lemma~\ref{prop:Bessel}]
In terms of
\[
F(t,\upsigma):=e^{\uplambda\upsigma_0}+\int_{\upsigma_0}^{\upsigma} e^{\uplambda\upsigma'}g_a(t,\upsigma';\upsigma_0)\,\mathrm d\upsigma',
\]
the Volterra--Hardy equation~\eqref{eq:Hardy-hom} for $g_a$ becomes
\[
\partial_t g_a(t,\upsigma;\upsigma_0)
-a\,\frac{\uplambda e^{-\uplambda \upsigma}}{1-e^{-\uplambda \upsigma}}\,F(t,\upsigma)=0.
\]
Since $g_a=e^{-\uplambda \upsigma}\partial_\upsigma F$, we arrive at
\[
\partial_t\partial_\upsigma F(t,\upsigma)-a\,\frac{\uplambda}{1-e^{-\uplambda \upsigma}}\,F(t,\upsigma)=0,
\qquad
F(t,\upsigma_0)=e^{\uplambda\upsigma_0},
\qquad
F(0,\upsigma)=e^{\uplambda\upsigma_0}.
\]
Now introduce
\[
y:=Y(\upsigma;\upsigma_0)=\int_{\upsigma_0}^{\upsigma} \frac{\uplambda\,\mathrm d\uprho}{1-e^{-\uplambda\uprho}}
=
\log \frac{e^{\uplambda \upsigma}-1}{e^{\uplambda\upsigma_0}-1},
\]
and set $U(t,y):=F(t,\upsigma(y))$. Since
\[
\partial_\upsigma y = \frac{\uplambda}{1-e^{-\uplambda \upsigma}},
\qquad
\partial_\upsigma F(t,\upsigma)=\frac{\uplambda}{1-e^{-\uplambda \upsigma}}\partial_yU(t,y),
\]
we find
\[
\partial_t\partial_\upsigma F(t,\upsigma)
=
\frac{\uplambda}{1-e^{-\uplambda \upsigma}}\partial_t\partial_yU(t,y).
\]
Hence the equation for $F$ reduces to the Goursat problem
\[
\partial_t\partial_yU(t,y)-a\,U(t,y)=0,
\qquad
U(t,0)=e^{\uplambda\upsigma_0},
\qquad
U(0,y)=e^{\uplambda\upsigma_0}.
\]
We solve this by Laplace transform in $t$. Writing
\[
\widetilde U(\upmu,y):=\int_0^\infty e^{-\upmu t}U(t,y)\,\mathrm d t,
\]
the transformed equation reads
\[
\upmu\,\partial_y\widetilde U(\upmu,y)-a\,\widetilde U(\upmu,y)=0,
\qquad
\widetilde U(\upmu,0)=\frac{e^{\uplambda\upsigma_0}}{\upmu}.
\]
Therefore
\[
\widetilde U(\upmu,y)=\frac{e^{\uplambda\upsigma_0}}{\upmu}e^{ay/\upmu}
=
e^{\uplambda\upsigma_0}\sum_{k\ge0}\frac{a^k y^k}{k!}\upmu^{-(k+1)}.
\]
Inverting termwise yields
\[
U(t,y)
=
e^{\uplambda\upsigma_0}\sum_{k\ge0}\frac{(aty)^k}{(k!)^2}
=
e^{\uplambda\upsigma_0}I_0 \left(2\sqrt{aty}\right).
\]
Differentiating and using $I_0'(z)=I_1(z)$ gives
\[
\partial_yU(t,y)
=
e^{\uplambda\upsigma_0}\sqrt{\frac{at}{y}}\,I_1 \left(2\sqrt{aty}\right).
\]
Returning to $F$ and then to $g_a$,
\[
g_a(t,\upsigma;\upsigma_0)
=
e^{-\uplambda \upsigma}\partial_\upsigma F(t,\upsigma)
=
e^{-\uplambda \upsigma}\frac{\uplambda}{1-e^{-\uplambda \upsigma}}\partial_yU(t,Y(\upsigma;\upsigma_0)),
\]
which is exactly \eqref{eq:Bessel-full}. The formula
\eqref{eq:Bessel-output} follows from the limit $\uplambda\downarrow0$, because
$Y(\upsigma;\upsigma_0)\to\log(\upsigma/\upsigma_0)$ and
$\frac{\uplambda e^{-\uplambda(\upsigma-\upsigma_0)}}{1-e^{-\uplambda \upsigma}}\to \upsigma^{-1}$.
\end{proof}

With Lemma~\ref{prop:Bessel} at hand, the expression for $\mathscr{S}_t(\upsigma_0)$ in~\eqref{eq:Sb-def} becomes explicit, and a direct analysis will prove its $\mathsf{U}$-shape.
Set
\[
c_n:=a_nt,
\qquad
Y(\upsigma_0):=Y(1;\upsigma_0)=\log \frac{e^\uplambda-1}{e^{\uplambda\upsigma_0}-1},
\]
and define
\[
\uppsi_c(y):=\sqrt{\frac{c}{y}}\,I_1(2\sqrt{cy})
=\sum_{k\ge0}\frac{c^{k+1}}{k!(k+1)!}y^k.
\]
At the output layer $\upsigma=1$, Lemma~\ref{prop:Bessel} yields
\[
g_{a_n}(t,1;\upsigma_0)
=
\left(\frac{\uplambda}{e^\uplambda-1}+\uplambda e^{-Y(\upsigma_0)}\right)\uppsi_{c_n}(Y(\upsigma_0)).
\]

\begin{proof}[Proof of Theorem~\ref{thm:U-shape}]
The map $\upsigma_0\mapsto Y(\upsigma_0)$ is a smooth decreasing bijection from $(0,1]$
onto $[0,\infty)$, so it suffices to consider
\[
\mathscr{H}(y):=
\sum_{n\ge1}e^{-\frac{\uppi^2}{2M^2}n^2}
\left(\frac{\uplambda}{e^\uplambda-1}+\uplambda e^{-y}\right)\uppsi_{c_n}(y),
\qquad y\in[0,\infty).
\]
For one mode, write
\[
h_c(y):=
\left(\frac{\uplambda}{e^\uplambda-1}+\uplambda e^{-y}\right)\uppsi_c(y).
\]
A direct differentiation yields
\[
h_c''(y)
=
\frac{\uplambda}{e^\uplambda-1}\uppsi_c''(y)
+
\uplambda e^{-y}\left(\uppsi_c''(y)-2\uppsi_c'(y)+\uppsi_c(y)\right).
\]
Since
\[
\uppsi_c(y)=\sum_{k\ge0}\frac{c^{k+1}}{k!(k+1)!}y^k,
\]
we have
\[
\uppsi_c''(y)-2\uppsi_c'(y)+\uppsi_c(y)
=
\sum_{k\ge0}
\frac{c^{k+1}}{k!(k+3)!}
\left(c^2-2(k+3)c+(k+2)(k+3)\right)y^k.
\]
If $0<c\le 3-\sqrt3$, then every coefficient in the last series is
nonnegative, hence $h_c''(y)>0$ on $[0,\infty)$. Under
\eqref{eq:affine-smallness}, each summand is therefore strictly convex, so
$\mathscr{H}$ is strictly convex.
At $y=0$ we have $\uppsi_c(0)=c$ and $\uppsi_c'(0)=c^2/2$, hence
\[
h_c'(0)
=
\uplambda c\left(\frac{c}{2(1-e^{-\uplambda})}-1\right).
\]
The second bound in \eqref{eq:affine-smallness} implies $h_c'(0)<0$
for every mode, hence $\mathscr{H}'(0)<0$. Finally, because
$I_1(z)\sim e^z/\sqrt{2\uppi z}$ as $z\to\infty$, each $h_c(y)$ tends
to $\infty$ as $y\to\infty$, and so does $\mathscr{H}(y)$. A strictly convex
function on $[0,\infty)$ with negative initial slope and diverging right tail
has a unique global minimizer. Pulling it back through the bijection
$\upsigma_0\mapsto Y(\upsigma_0)$ proves the claim.
\end{proof}

\section{Extensions}
\label{sec:extensions}
\subsection{Numerical evidence for~\eqref{eq:full-dyn}}\label{sec:highdim-numerics}

To test whether the primacy--recency mechanism persists beyond our
reduced model~\eqref{eq:gpt-dyn-alibi}, we simulated the full
dynamics~\eqref{eq:full-dyn}. Let $B$ be a standard real Brownian motion,
independent of the initial data, and set $\mathrm dZ_t=-Z_t\,\mathrm dt+\mathrm dB_t$ with $Z_0=0$ and $\upeta_t=0.15(1+\tanh Z_t)\in(0,0.3)$. Conditional on the path $Z$, the query, key, and value matrices
under $\uprho(t)$ are independent copies of
\[
\frac{\mathrm{Id}+\upeta_t H}
     {\|\mathrm{Id}+\upeta_t H\|_{\mathrm{op}}},
\]
where $H$ is obtained by symmetrizing a matrix with iid standard
Gaussian entries and rescaling it to have operator norm one.
Thus the strength of the near-identity perturbations fluctuates
randomly with depth. For each realized path, we evaluate the matrix
expectations in both the numerator and denominator
of~\eqref{eq:full-dyn} conditionally on that path and evolve the
resulting dynamics.

In this higher-dimensional setting, we replace the angular soft accuracy~\eqref{eq:Acc-soft} by its natural rotationally invariant analogue
\begin{equation*}
{\mathscr A}_{N}(t,\upsigma_0)=\E e^{-|x_N(t)-x_{i_0}(0)|^2}=\E e^{-2(1-\langle x_N(t),x_{i_0}(0)\rangle)},\qquad i_0=\lceil\upsigma_0 N\rceil.
\end{equation*}
For iid uniform initial data, similarly as in~\eqref{eq:soft-accuracy-expansion}, the mean-field score obtained by replacing $x_{i_0}(0)$ with an independent uniform source token is a constant, equal to
\[c_d:=e^{-2}\Gamma\Big(\frac d2\Big)I_{\frac d2-1}(2).\]
We therefore isolate the correlation contribution by writing
\begin{equation}\label{eq:mod-score}
{\mathscr A}_{N}(t,\upsigma_0)=c_d+\frac1N{\mathscr S}_{N,t}(\upsigma_0).
\end{equation}
so that $\mathscr S_{N,t}$ captures the source-position dependence at the natural $O(N^{-1})$ correlation scale.
The top-left panel of Figure~\ref{fig:alibi-highdim} displays $\mathscr S_{N,t}$ and shows a clear $\mathsf U$-shaped profile, providing numerical evidence that the primacy--recency mechanism persists for the full high-dimensional dynamics~\eqref{eq:full-dyn}, beyond the scope of our present analysis.

\subsection{Near-deterministic initial data}\label{sec:near-monokinetic}

The iid uniform state is one of the rare regimes in which the limiting correlation equation diagonalizes in the angular variable.  We now show that the same mechanism reappears in a seemingly opposite regime, namely for small random perturbations of a deterministic profile.
Let us first specify the initial data. Fix a profile $h_\circ\in C^1([0,1];\T)$
and a $W_1$-Lipschitz map $[0,1]\to\Pc(\R):\upsigma\mapsto\upnu_\upsigma$ such that, for some $K<\infty$,
\begin{equation}\label{eq:monokinetic-noise}
\supp\upnu_\upsigma\subset[-K,K],
\qquad
\int_\R \upxi\,\upnu_\upsigma(\mathrm d\upxi)=0,
\qquad
v(\upsigma):=\int_\R \upxi^2\,\upnu_\upsigma(\mathrm d\upxi).
\end{equation}
For each $N$, let $(\upxi_j)_{1\le j\le N}$ be independent random variables with laws $\upxi_j\sim\upnu_{j/N}$, and let $\uptheta_j^\upvarepsilon$ solve~\eqref{eq:gpt-dyn-alibi} with initial condition
\begin{equation}\label{eq:near-monokinetic-particle-init}
\uptheta_j^\upvarepsilon(0)
=
h_\circ(j/N)+\upvarepsilon\upxi_j
\pmod{2\uppi}.
\end{equation}
In other words, the initial configuration is an $O(\upvarepsilon)$ random perturbation of the deterministic profile $h_\circ$, with position-dependent noise of variance $\upsigma\mapsto v(\upsigma)$.
For $\upvarepsilon=0$, the mean-field solution is monokinetic,
\[
f^0(t,\upsigma)=\updelta_{h(t,\upsigma)},
\]
where $h$ is the unique solution of
\begin{equation}\label{eq:monokinetic-flow}
\begin{dcases}
\partial_t h(t,\upsigma)
=
\int_0^\upsigma
\mathsf k_{\uplambda}(\upsigma,\uptau)
\mathsf w_{\upbeta}'\big(
h(t,\upsigma)-h(t,\uptau)
\big)\,\mathrm d\uptau,
\\
h(0,\upsigma)=h_\circ(\upsigma).
\end{dcases}
\end{equation}
Starting from the $O(\upvarepsilon)$ perturbation~\eqref{eq:near-monokinetic-particle-init}, we show that the first nontrivial correlations occur precisely at order~$O(\upvarepsilon^2)$.

\begin{prop}\label{prop:near-monokinetic}
For all $\upvarphi,\uppsi\in C^\infty(\T)$, all $t\ge0$, and
$0<\upsigma_0<\upsigma\le1$,
\begin{eqnarray}
\lim_{\upvarepsilon\downarrow0}\lim_{N\uparrow\infty}\frac1{\upvarepsilon^2}\Cov\Big(\uppsi(\uptheta_{\lceil N\upsigma\rceil}^\upvarepsilon(t)),\upvarphi(\uptheta_{\lceil N\upsigma\rceil}^\upvarepsilon(0))\Big)&=&v(\upsigma)D(t,\upsigma)\uppsi'(h(t,\upsigma))\upvarphi'(h_\circ(\upsigma)),\label{eq:monokinetic-autocorrelation-limit}\\
\lim_{\upvarepsilon\downarrow0}\lim_{N\uparrow\infty}\frac N{\upvarepsilon^2}\Cov\Big(\uppsi(\uptheta_{\lceil N\upsigma\rceil}^\upvarepsilon(t)),\upvarphi(\uptheta_{\lceil N\upsigma_0\rceil}^\upvarepsilon(0))\Big)
&=&v(\upsigma_0)R(t,\upsigma;\upsigma_0)\uppsi'(h(t,\upsigma))\upvarphi'(h_\circ(\upsigma_0)),\label{eq:monokinetic-cross-correlation-limit}
\end{eqnarray}
where we have set
\begin{gather*}
B(t,\upsigma,\uptau)
:=
\mathsf k_{\uplambda}(\upsigma,\uptau)
\mathsf w_{\upbeta}''\big(
h(t,\upsigma)-h(t,\uptau)
\big),
\\
L(t,\upsigma):=\int_0^\upsigma B(t,\upsigma,\uptau)\,\mathrm d\uptau,\qquad
D(t,\upsigma):=\exp\left(\int_0^tL(s,\upsigma)\,\mathrm ds\right),
\end{gather*}
and where for $0<\upsigma_0<\upsigma\le1$ we denote by $R(t,\upsigma;\upsigma_0)$ the unique solution of
\begin{equation}\label{eq:monokinetic-R}
\begin{dcases}
\partial_tR(t,\upsigma;\upsigma_0)
=L(t,\upsigma)R(t,\upsigma;\upsigma_0)
-\int_{\upsigma_0}^{\upsigma}B(t,\upsigma,\uptau)R(t,\uptau;\upsigma_0)\,\mathrm d\uptau
-B(t,\upsigma,\upsigma_0)D(t,\upsigma_0),\\
R(0,\upsigma;\upsigma_0)=0.
\end{dcases}
\end{equation}
Those convergences~\eqref{eq:monokinetic-autocorrelation-limit}--\eqref{eq:monokinetic-cross-correlation-limit} are locally uniform in time $t$ and in the macroscopic position $\upsigma$, away from~$\upsigma=0$ and from the diagonal $\upsigma=\upsigma_0$.
\end{prop}

\begin{proof}
The assumptions on $h_\circ$ and $\upnu_\upsigma$ ensure the initial convergence condition~\eqref{eq:init-conv} for every fixed~$\upvarepsilon$, with limit distribution $f_\circ^\upvarepsilon$ that is the law of $h_\circ(\upsigma)+\e\upxi_\upsigma$ with $\upxi_\upsigma\sim\upnu_\upsigma$.
Hence Theorem~\ref{th:lost} reduces the two limits in~\eqref{eq:monokinetic-autocorrelation-limit}--\eqref{eq:monokinetic-cross-correlation-limit} to identifying the small-$\upvarepsilon$ asymptotics of the corresponding limiting autocorrelation $A_\upvarphi^\upvarepsilon$ and cross-correlation $C_\upvarphi^\upvarepsilon$ solving~\eqref{eq:Aphi-lambda}--\eqref{eq:Cphi-lambda}.

We start with the autocorrelation. Let $Y^\upvarepsilon(t,\upsigma,\upxi)$ be the characteristic associated with the mean-field transport equation~\eqref{eq:mfl-lambda}, issued from $h_\circ(\upsigma)+\upvarepsilon \upxi$:
\begin{equation}\label{eq:near-monokinetic-characteristics}
\begin{dcases}
\partial_tY^\upvarepsilon(t,\upsigma,\upxi)
=\int_0^\upsigma\mathsf k_{\uplambda}(\upsigma,\upsigma')\int_\R \mathsf w_{\upbeta}'\big(Y^\upvarepsilon(t,\upsigma,\upxi)-Y^\upvarepsilon(t,\upsigma',\upxi')\big)\,\upnu_{\upsigma'}(\mathrm d\upxi')\,\mathrm d\upsigma',\\
Y^\upvarepsilon(0,\upsigma,\upxi)=h_\circ(\upsigma)+\upvarepsilon\upxi.
\end{dcases}
\end{equation}
Then $f^\upvarepsilon(t,\upsigma)=Y^\upvarepsilon(t,\upsigma,\cdot)_\#\upnu_\upsigma$.
Since the noise is uniformly bounded and $\mathsf w_\upbeta$ is smooth,
standard parameter dependence for the characteristic equation shows that
$Y^\upvarepsilon$ is twice differentiable with respect to
$\upvarepsilon$, locally uniformly in $(t,\upsigma,\upxi)$ away from
$\upsigma=0$. Differentiating~\eqref{eq:near-monokinetic-characteristics} at $\upvarepsilon=0$, we find
\begin{equation}\label{eq:near-monokinetic-linearized-characteristics}
\partial_t\eta(t,\upsigma,\upxi)
=
L(t,\upsigma)\eta(t,\upsigma,\upxi)
-
\int_0^\upsigma
B(t,\upsigma,\uptau)\bar\eta(t,\uptau)\,\mathrm d\uptau,
\qquad
\eta(0,\upsigma,\upxi)=\upxi,
\end{equation}
in terms of
\[\eta(t,\upsigma,\upxi):=\left.\partial_\upvarepsilon
Y^\upvarepsilon(t,\upsigma,\upxi)\right|_{\upvarepsilon=0},
\qquad\bar\eta(t,\upsigma):=\int_\R\eta(t,\upsigma,\upxi)\,\upnu_\upsigma(\mathrm d\upxi).\]
Averaging in $\upxi$ and using the centering~\eqref{eq:monokinetic-noise}, we obtain the triangular homogeneous equation
\[\partial_t\bar\eta(t,\upsigma)=L(t,\upsigma)\bar\eta(t,\upsigma)-\int_0^\upsigma B(t,\upsigma,\uptau)\bar\eta(t,\uptau)\,\mathrm d\uptau,
\qquad\bar\eta(0,\upsigma)=0.\]
Since its unique solution is $\bar\eta\equiv0$, we conclude
\[\partial_\upvarepsilon Y^\upvarepsilon(t,\upsigma,\upxi)|_{\upvarepsilon=0}=D(t,\upsigma)\upxi,\]
hence
\begin{equation}\label{eq:near-monokinetic-characteristic-expansion}
Y^\upvarepsilon(t,\upsigma,\upxi)=h(t,\upsigma)+\upvarepsilon D(t,\upsigma)\upxi+O(\upvarepsilon^2),
\end{equation}
locally uniformly.
Since $A_\upvarphi^\upvarepsilon$ is the transport by
$Y^\upvarepsilon$ of the initial covariance measure, for every test
function $\uppsi$ we have
\[\langle A_\upvarphi^\upvarepsilon(t,\upsigma),\uppsi\rangle=\Cov_{\upxi\sim\upnu_\upsigma}\Big(\uppsi(Y^\upvarepsilon(t,\upsigma,\upxi)),\upvarphi(h_\circ(\upsigma)+\upvarepsilon\upxi)\Big).\]
Using \eqref{eq:near-monokinetic-characteristic-expansion} and the
centering of $\upnu_\upsigma$ yields
\[\upvarepsilon^{-2}\left\langle A_\upvarphi^\upvarepsilon(t,\upsigma),\uppsi\right\rangle \to v(\upsigma)D(t,\upsigma)\uppsi'(h(t,\upsigma))\upvarphi'(h_\circ(\upsigma)),\]
thus proving~\eqref{eq:monokinetic-autocorrelation-limit}.

We turn to the cross-correlation. Dividing~\eqref{eq:Cphi-lambda} by $\upvarepsilon^2$, using $f^\upvarepsilon\to\updelta_h$ and~\eqref{eq:monokinetic-autocorrelation-limit}, we deduce that every limit point~$\mathcal C_\upvarphi$ of~$\upvarepsilon^{-2}C_\upvarphi^\upvarepsilon$ solves the equation obtained from~\eqref{eq:Cphi-lambda} by replacing $f$ with $\updelta_h$ and the autocorrelation~$A_\upvarphi(t,\upsigma)$ with its limit $-v(\upsigma)D(t,\upsigma)\upvarphi'(h_\circ(\upsigma))\updelta'_{h(t,\upsigma)}$.
Using the distributional identities
\[
\mathsf w_\upbeta'\ast\updelta'_a
=
\mathsf w_\upbeta''(\,\cdot-a),
\qquad
b\,\updelta'_a
=
b(a)\updelta'_a-b'(a)\updelta_a,
\qquad
\partial_t\updelta'_{h(t)}
=
-\partial_th(t)\updelta''_{h(t)},
\]
a direct substitution shows that
\[\mathcal C_\upvarphi(t,\upsigma,\cdot;\upsigma_0)=-v(\upsigma_0)R(t,\upsigma;\upsigma_0)\upvarphi'(h_\circ(\upsigma_0))\updelta'_{h(t,\upsigma)}\]
solves the limiting equation if and only if $R$ satisfies
\eqref{eq:monokinetic-R}.  The coefficients of
$\updelta''_{h(t,\upsigma)}$ cancel precisely by
\eqref{eq:monokinetic-flow}.  Uniqueness of the triangular linear
Volterra equation then gives the convergence of the whole family and,
after testing against $\uppsi$, proves
\eqref{eq:monokinetic-cross-correlation-limit}.
\end{proof}

In the homogeneous single-cluster case, the above description of the cross-correlation reduces naturally to the same Volterra--Hardy equation as in the iid uniform regime.

\begin{cor}
\label{cor:homogeneous-cluster}
In the setting of Proposition~\ref{prop:near-monokinetic}, assume that
$h_\circ\equiv h_*$ is constant, and set
\begin{equation*}
a_*:=-\mathsf w_\upbeta''(0)=\upbeta e^\upbeta>0.
\end{equation*}
Then $h(t,\upsigma)\equiv h_*$,
$D(t,\upsigma)=e^{-a_*t}$, and
\begin{equation}\label{eq:homogeneous-cluster-R}
R(t,\upsigma;\upsigma_0)
=
e^{-a_*t}g_{a_*}(t,\upsigma;\upsigma_0),
\end{equation}
where $g_{a_*}$ is the solution of \eqref{eq:Hardy-hom}.
Assuming in addition that $\upnu_\upsigma\equiv\upnu$ is independent of $\upsigma$, with variance~$v_*>0$, we decompose the soft accuracy~\eqref{eq:Acc-soft-Fourier} into its independent-prompt baseline and its correlation contribution: more precisely,
\begin{equation}\label{eq:defSNeps}
{\mathscr A}_N^\upvarepsilon(t,\upsigma_0)={\mathscr B}_N^\upvarepsilon(t)+\upvarepsilon^2\frac{\sqrt{2\uppi}}{MN}{\mathscr S}_{N,t}^\upvarepsilon(\upsigma_0),
\end{equation}
where ${\mathscr B}_N^\upvarepsilon(t)$ is independent of $\upsigma_0$ and where
\[
{\mathscr S}_{N,t}^\upvarepsilon(\upsigma_0):=
\frac{N}{2\upvarepsilon^2}
\sum_{n\in\Z}
e^{-\frac{\uppi^2}{2M^2}n^2}
\Cov\left(e^{in\uptheta_N(t)},e^{-in\uptheta_{i_0}(0))}\right).\]
Then we have
\begin{align}
\lim_{\upvarepsilon\downarrow0}
\lim_{N\uparrow\infty}
{\mathscr S}_{N,t}^\upvarepsilon(\upsigma_0)
&=
v_*e^{-a_*t}
\left(
\sum_{n\ge1}
n^2e^{-\frac{\uppi^2}{2M^2}n^2}
\right)
g_{a_*}(t,1;\upsigma_0).
\label{eq:homogeneous-cluster-soft-accuracy}
\end{align}
Moreover, if $\uplambda>0$, $t>0$, and
\begin{equation}\label{eq:homogeneous-cluster-smallness}
ta_*
<
\min\left\{
3-\sqrt3,\,
2(1-e^{-\uplambda})
\right\},
\end{equation}
the right-hand side of
\eqref{eq:homogeneous-cluster-soft-accuracy}, as a function of
$\upsigma_0\in(0,1)$, has the same primacy--recency
$\mathsf U$-shape as in Theorem~\ref{thm:U-shape}, with a unique global
minimum.
\end{cor}

\begin{proof}
Since
$\int_0^\upsigma
\mathsf k_{\uplambda}(\upsigma,\upsigma')\,\mathrm d\upsigma'=1$
and $\mathsf w_\upbeta'(0)=0$, we have $h\equiv h_*$ and
$D=e^{-a_*t}$.  Equation~\eqref{eq:monokinetic-R} becomes
\[
\partial_tR
=
-a_*R
+
a_*\left(
\mathsf k_{\uplambda}(\upsigma,\upsigma_0)D
+
\int_{\upsigma_0}^{\upsigma}
\mathsf k_{\uplambda}(\upsigma,\uptau)
R(t,\uptau;\upsigma_0)\,\mathrm d\uptau
\right).
\]
Substituting $R=e^{-a_*t}g$ gives exactly
\eqref{eq:Hardy-hom} with $a=a_*$, and proves
\eqref{eq:homogeneous-cluster-R}.
Applying~\eqref{eq:monokinetic-cross-correlation-limit} with $\upvarphi(\uptheta)=e^{-in\uptheta}$ and $\uppsi(\uptheta)=e^{in\uptheta}$, and noting that $\upvarphi'(h_*)\uppsi'(h_*)=n^2$, we obtain
\begin{equation*}
\lim_{\upvarepsilon\downarrow0}
\lim_{N\uparrow\infty}
\frac N{\upvarepsilon^2}
\Cov\left(
e^{in\uptheta_{\lceil N\upsigma\rceil}^\upvarepsilon(t)},
e^{-in\uptheta_{\lceil N\upsigma_0\rceil}^\upvarepsilon(0)}
\right)=v(\upsigma_0)n^2e^{-a_*t}
g_{a_*}(t,\upsigma;\upsigma_0).
\end{equation*}
The conclusion~\eqref{eq:homogeneous-cluster-soft-accuracy} then follows by summing the
positive and negative Fourier modes in
\eqref{eq:Acc-soft-Fourier}.  Finally, the proof of Theorem~\ref{thm:U-shape}, applied to the single parameter $a_*$, yields the last assertion under~\eqref{eq:homogeneous-cluster-smallness}.
\end{proof}

\begin{rem}
For nonconstant $h_\circ$, the response equation
\eqref{eq:monokinetic-R} has coefficients depending on both time and
position and does not in general reduce to \eqref{eq:Hardy-hom}.
Likewise, even for constant $h_\circ$, a nonconstant variance profile
multiplies the leading response by $v(\upsigma_0)$ and can alter its shape.
Thus Proposition~\ref{prop:near-monokinetic} provides a perturbative
framework for clustered initializations, while the scalar
$\mathsf U$-shape conclusion of
Corollary~\ref{cor:homogeneous-cluster} is specific to a homogeneous
single cluster with position-independent perturbations.
\end{rem}

\begin{figure}[h!]
\centering
\includegraphics[height=4.8cm]{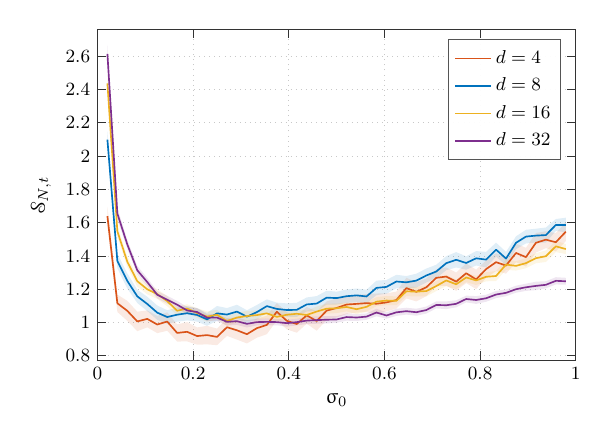}\hspace{.04\textwidth}
\includegraphics[height=4.8cm]{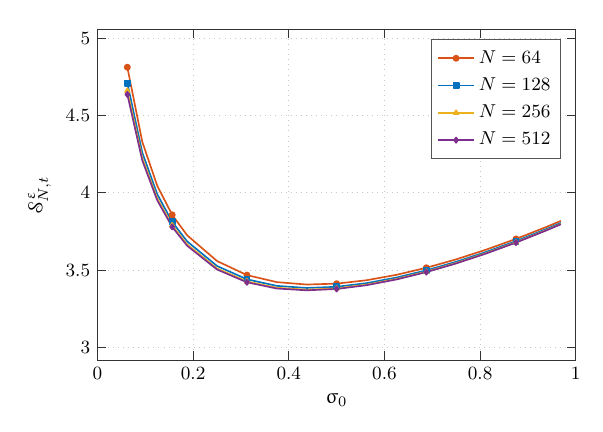}
\par\medskip
\includegraphics[height=4.8cm]{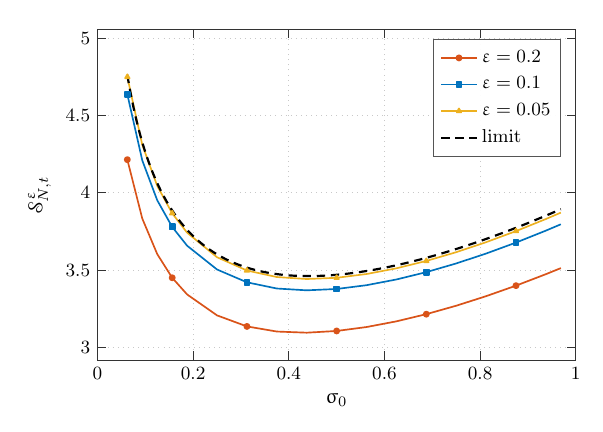}\hspace{.04\textwidth}
\includegraphics[height=4.8cm]{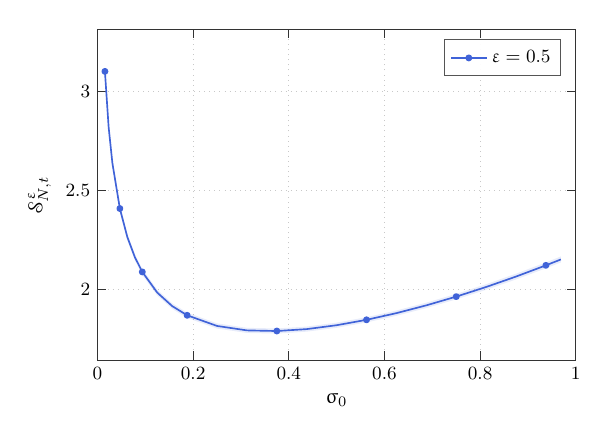}
\caption{\emph{Top left:} $\mathscr S_{N,t}$,
defined in~\eqref{eq:mod-score}, for the full
dynamics~\eqref{eq:full-dyn} with the randomly time-dependent
near-identity matrix law described in
Section~\ref{sec:highdim-numerics}. Here $N=48$,
$\upbeta=4$, $\uplambda=2$, $t=2$.
\emph{Top right:} $\mathscr S_{N,t}^{\upvarepsilon}$,
defined in~\eqref{eq:defSNeps}, for
$N=64,128,256,512$ at $\upvarepsilon=0.1$.
\emph{Bottom left:} $\mathscr S_{N,t}^{\upvarepsilon}$
for $\upvarepsilon=0.2,0.1,0.05$ at $N=512$;
the dashed curve is the limiting
expression~\eqref{eq:homogeneous-cluster-soft-accuracy}
with $v_*=1/3$.
\emph{Bottom right:} $\mathscr S_{N,t}^{\upvarepsilon}$
for $N=512$ and $\upvarepsilon=0.5$.
These last three panels use the reduced
dynamics~\eqref{eq:gpt-dyn-alibi}, with
$\upbeta=\uplambda=1$, $M=8$, $t=0.3$, and iid initial
angles uniform on $[-\upvarepsilon,\upvarepsilon]$.}\label{fig:alibi-highdim}
\end{figure}

\subsection*{Acknowledgments}

The authors thank Thierry Paul for organizing several inspiring ``Round Meanfield'' workshops over the past years, as well as Emmanuel Tr\'elat and Pierre Le Bris for motivating discussions, which resulted in the genesis of this work.

MD acknowledges financial support from the European Union (ERC, PASTIS, Grant Agreement n$^\circ$101075879).\footnote{{Views and opinions expressed are however those of the authors only and do not necessarily reflect those of the European Union or the European Research Council Executive Agency. Neither the European Union nor the granting authority can be held responsible for them.}} BG's research was supported by a Sorbonne Emergences grant and a gift from Google.

\subsection*{AI tool disclosure}

ChatGPT was used for proofreading, generating the plots, and to suggest the main clues behind the explicit computation of Lemma~\ref{prop:Bessel}. Outside of these AI tool uses, the text of this paper was human generated.

\bibliographystyle{plain}
\bibliography{biblio}

\bigskip
\bigskip

\begin{minipage}[t]{.5\textwidth}
    {\footnotesize{\bf Mitia Duerinckx}\par
      D\'epartement de Math\'emathiques\par
Universit\'e Libre de Bruxelles\par
Boulevard du Triomphe\par
B-1050 Bruxelles, Belgium
     \par
      e-mail: \href{mailto:mitia.duerinckx@ulb.be}{\textcolor{blue}{\scriptsize mitia.duerinckx@ulb.be}}
      }
    \end{minipage}
    \begin{minipage}[t]{.5\textwidth}
      {\footnotesize{\bf Borjan Geshkovski}\par
      Laboratoire Jacques-Louis Lions\par
      Inria \& Sorbonne Universite\par
      4 Place Jussieu\par
      75005 Paris, France\par
     \par
      e-mail: \href{mailto:borjan.geshkovski@inria.fr}{\textcolor{blue}{\scriptsize borjan.geshkovski@inria.fr}}
      }
    \end{minipage}

    \bigskip
    \bigskip

    \begin{center}
        \begin{minipage}[t]{.5\textwidth}
    {\footnotesize{\bf Stefano Rossi}\par
      Dipartimento di Matematica Guido Castelnuovo\par
Sapienza Universit\`a di Roma\par
Piazzale Aldo Moro 5\par
00185 Rome, Italy
     \par
      e-mail: \href{mailto:stefano.rossi2@uniroma1.it}{\textcolor{blue}{\scriptsize stefano.rossi2@uniroma1.it}}
      }
    \end{minipage}%
    \end{center}

\end{document}